\documentclass[reqno]{amsart}
\usepackage{amssymb}
\usepackage{amsmath}
\usepackage[mathscr]{euscript}

\usepackage[small]{caption}
\usepackage{psfrag}
\usepackage{graphicx}
\usepackage{color}

\makeatletter
\@addtoreset{equation}{section}
\makeatother

\newtheorem{theorem}{Theorem}[section]
\newtheorem{lemma}[theorem]{Lemma}
\newtheorem{proposition}[theorem]{Proposition}
\newtheorem{corollary}[theorem]{Corollary}

\newtheorem{definition}[theorem]{Definition}
\newtheorem{remark}[theorem]{Remark}
\newtheorem{example}[theorem]{Example}

\newcommand{\mc}[1]{{\mathcal #1}}
\newcommand{\mf}[1]{{\mathfrak #1}}
\newcommand{\mb}[1]{{\mathbf #1}}
\newcommand{\bb}[1]{{\mathbb #1}}

\newcommand{\ms}[1]{{\mathscr #1}}

\newcounter{as}

\newtheorem{asser}[as]{Assertion}

\begin{document}

\title[]{A topology for limits of Markov chains}

\author{C. Landim}

\address{\noindent IMPA, Estrada Dona Castorina 110, CEP 22460 Rio de
  Janeiro, Brasil and CNRS UMR 6085, Universit\'e de Rouen, Avenue de
  l'Universit\'e, BP.12, Technop\^ole du Madril\-let, F76801
  Saint-\'Etienne-du-Rouvray, France.  \newline e-mail: \rm
  \texttt{landim@impa.br} }


\begin{abstract}
  In the investigation of limits of Markov chains, the presence of
  states which become instantaneous states in the limit may prevent
  the convergence of the chain in the Skorohod topology. We present in
  this article a weaker topology adapted to handle this situation. We
  use this topology to derive the limit of random walks among random
  traps and sticky zero-range processes.
\end{abstract}

\maketitle

\section{Introduction}
\label{esec1}

Some Markov chains can be approximated by Markovian dynamics evolving
in a contracted state space.  This is the case of certain zero-range
models, whose dynamics can be approximated by the one of a random
walk, \cite{bl3, l1}, and of some polymer models whose evolution can
be approximated by a two-state Markov chain \cite{clmst, cmt, bl9,
  lt1}. We proposed in \cite{bl2, bl7} a formal definition of this
phenomenon.  In analogy to statistical mechanic models, we named these
processes metastable Markov chains, and we developed tools to prove
the convergence (of the projection on a contracted state space) of
these chains to Markovian dynamics. The erratic behavior of the
projected chain in very short time intervals precludes convergence in
the Skorohod topology. We introduce in this article a topology in
which the convergence takes place. This topology might be useful in
other contexts. For example, in the investigation of limits of Markov
chains when some states become instantaneous states in the limit, that
is, states whose jump rates become infinite. To explain the
topological problem created by the asymptotic instantaneous states and
to present the main results of the article, we examine in this
introduction sticky zero-range processes and random walks among random
traps.

\smallskip
\noindent{\bf 1. Zero-range processes.}
Fix a finite set $S_L =\{1, \dots, L\}$, and denote by $E_{L,N}$,
$N\ge 1$, the configurations obtained by distributing $N$ particles on
$S_L$:
\begin{equation*}
E_{L,N}\;:=\;\big\{ \eta\in\bb N^{S_L} : \sum_{x\in S_L} \eta_x = N \big\}\;. 
\end{equation*}

Consider an irreducible, continuous-time random walk $\{Z(t) : t\ge
0\}$ on $S_L$ which jumps from $x$ to $y$ at a rate $r(x,y)$ which is
reversible with respect to the uniform measure, $r(x,y) = r(y,x)$,
$x$, $y \in S_L$.  Fix $\alpha>1$, and let $g:\bb N\to \bb R$ be given
by
\begin{equation*}
g(0)=0\; , \quad g(1)=1\;, \quad
\textrm{and}\quad g(n) \;=\; \Big( \frac{n}{n-1} \Big)^\alpha\;,
\; n\ge 2\;,
\end{equation*}
so that $\prod_{i=1}^{n}g(i)=n^{\alpha}$, $n\ge 1$. 

Denote by $\{\eta^N(t) : t\ge 0\}$ the zero-range process on $S_L$ in
which a particle jumps from a site $x$, occupied by $k$ particles, to
a site $y$ at rate $g(k) r(x,y)$. The generator of this Markov chain
$\eta(t) = \eta^N(t)$, represented by $L_N$, acts on functions $F:
E_{L,N}\to\bb R$ as
\begin{equation*}
(L_N F) (\eta) \;=\; \sum_{\stackrel{x,y\in S_L}{x\not = y}}
g(\eta_x) \, r(x,y) \, \big\{ F(\sigma^{x,y}\eta) - F(\eta) \big\} \;,
\end{equation*}
where $\sigma^{x,y}\eta$ is the configuration obtained from $\eta$ by
moving a particle from $x$ to $y$:
\begin{equation*}
(\sigma^{x, y}\eta)_z\;=\;\left\{
\begin{array}{ll}
\eta_x-1 & \textrm{for $z=x$} \\
\eta_y+1 & \textrm{for $z=y$} \\
\eta_z & \rm{otherwise}\;. \\
\end{array}
\right. 
\end{equation*}

Fix a sequence $\{\ell_N : N\ge 1\}$ such that $1\ll \ell_N \ll N$,
where $a_N\ll b_N$ means that $a_N/b_N \to 0$.  For $x$ in $S_L$, let
\begin{equation*}
\ms E^x_N  \;:=\; \big\{\eta\in E_{L,N} : \eta_x \ge N - \ell_N \big\}\;.
\end{equation*}
Since $\ell_N/N\to 0$, on each set $\ms E^x_N$ the proportion of
particles at $x\in S_L$, $\eta_x/N$, is almost one. Assume that $N$ is
large enough so that $\ms E^x_N \cap \ms E^y_N = \varnothing$ for
$x\not = y$, and consider the partition
\begin{equation*}
E_{L,N} \;=\; \ms E^1_N \;\cup\; \cdots \;\cup\; \ms E^L_N\; \cup \;
\Delta_N\;,
\end{equation*}
where $\Delta_N$ is the set of configurations which do not belong to
the set $\ms E_N = \cup_{1\le x\le N} \ms E^x_N$.

Denote by $\pi_N$ the stationary measure of the zero-range dynamics
$\eta(t)$. We proved in \cite{bl3} that the measure $\pi_N$ is
concentrated on the set $\ms E_N$:
\begin{equation*}
\lim_{N\to\infty} \pi_N(\ms E^x_N) \;=\; \frac 1L \;\cdot
\end{equation*}
For $N>L$, let the projection $\Psi_N : E_{L,N} \to \{1,\dots, L\}
\cup\{N\}$ be defined by
\begin{equation*}
\Psi_N (\eta) \;=\; \sum_{x=1}^L x \, \mb 1\{\eta\in\ms E^x_N\} \;+\;
N \mb 1\{\eta\in\Delta_N\}\;,
\end{equation*}
and let $X_N(t) = \Psi_N(\eta(t))$, $\bb X_N(t) = X_N(t\theta_N)$,
where $\theta_N=N^{1+\alpha}$. By analogy with statistical mechanics
models, we call $\Psi_N (\eta)$ the order parameter. Note that
$X_N(t)$ is not a Markov chain.



A typical trajectory of $\bb X_N(t)$ is represented in Figure
\ref{efig2}.  The intervals $I_1$, $I_2, \dots$ correspond to the
sojourns of the rescaled process $\eta(t \theta_N)$ in a set $\ms
E^x_N$, while the time intervals $R_1$, $R_2, \dots$ correspond to the
excursions in the set $\Delta_N$. When the process $\eta (t)$ reaches
$\Delta_N$ from $\ms E^x_N$, a strong drift drives it back to $\ms
E^x_N$. With a very small probability it crosses $\Delta_N$ and
reaches a new set $\ms E^y_N$, $y\not = x$, before hitting $\ms
E^x_N$, and with a probability close to $1$ it returns to $\ms E^x_N$.
In this latter case, $\eta(t)$ remains close to the boundary between
$\ms E^x_N$ and $\Delta_N$ for an interval of time, very short
compared to the time it stays in the set $\ms E^x_N$, and crosses this
boundary a certain number of times until it is absorbed again in the
set $\ms E^x_N$. This explains the erratic behavior of $\bb X_N(t)$ in
the time intervals $R_j$.



\begin{figure}[htb]
  \centering
  \def\svgwidth{300pt}
  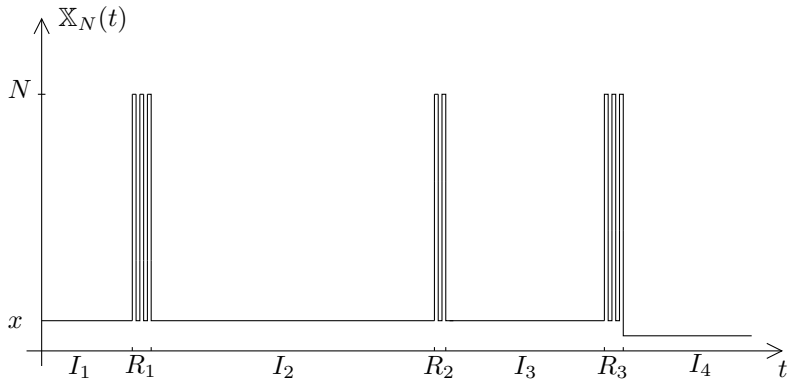
\caption{A typical trajectory of $\bb X_N(t)$.}
\label{efig2}
\end{figure}

We examined in \cite{bl3, bl7} the limit behavior of the projected
process $\bb X_N(t)$. The short excursions of $\eta^N(\theta_N t)$ in
$\Delta_N$ which correspond to short visits of $\bb X_N(t)$ to $N$
prevent the convergence of $\bb X_N(t)$ in the Skorohod topology to a
$S_L$-valued process. The same phenomenon occurs in random walks among
random traps.

\smallskip\noindent{\bf 2. Random walk among random traps.}  Denote by
$\bb T^d_N$ the discrete $d$-dimensional torus with $N^d$ points,
$d\ge 2$. We denote the sites of $\bb T^d_N$ by the letters $x$, $y$,
$z$. 

Note that the letters $x$, $y$ are used in this article to represent
three different objects. In the zero-range model, we use $x$, $y$ to
index the sets $\ms E^x_N$. In the random walk model, $x$, $y$ stand
for the sites of the torus $\bb T^d_N$, and in Sections \ref{esec5},
\ref{esec2}, \ref{esec3}, $x$, $y$ denote trajectories in paths spaces
such as $D([0,T], S_L)$.

Consider a sequence $\{W_j : j\ge 1\}$ of summable, non-increasing,
positive numbers:
\begin{equation*}
W_j\; \ge \; W_{j+1}\; >\; 0\;, \quad j\;\ge\; 1\;, \quad
\sum_{j\ge 1} W_j \;<\; \infty\;.
\end{equation*}
For each $N\ge 1$, denote by $\Psi_N : \bb T^d_N \to \{1, \dots,
V_N\}$, $V_N = |\bb T^d_N|= N^d$, a random uniform enumeration of the
sites of $\bb T^d_N$ and let $x^N_j = \Psi^{-1}_N(j)$. Let
\begin{equation*}
W^N_{x} \;=\; W_{\Psi_N(x)} \;, \quad x\in\bb T^d_N\;,
\end{equation*}
and let $\{\eta^N(t) : t\ge 0\}$ be the random walk on $\bb T^d_N$
which waits a mean $W^N_{x}$ exponential time at site $x$.  The
generator $L_N$ of this random walk is given by
\begin{equation*}
(L_N f)(x) \;=\; \frac 1{2d\, W^N_x} \sum_{y \sim x}
[f(y)-f(x)]\;, 
\end{equation*}
where the summation is carried over all nearest-neighbor sites $y$
of $x$. 

Denote by $\pi_N$ the stationary state of $\eta^N(t)$, $\pi_N(x) =
Z^{-1}_N W^N_x$, where $Z_N$ is the normalizing constant $Z_N =
\sum_{1\le j\le V_N} W_j$. Note that the stationary state is
concentrated on the first sites $x^N_1, \dots, x^N_L$. More precisely,
for every $\epsilon >0$, there exists $L\ge 1$ such that for all $N\ge
L$, $\pi_N(\{x^N_1, \dots, x^N_L\})\ge 1-\epsilon$.

To describe the asymptotic behavior of the Markov chain $\eta^N(t)$,
since the geometry of the graph $\bb T^d_N$ is lost as sites are
reshuffled for each $N$, we will examine the chain $X_N (t) = \Psi_N
(\eta^N(t))$ instead of the chain $\eta^N(t)$. 

Denote by $D(\bb R_+, \bb T^d_N)$ the set of trajectories $\mf x: \bb
R_+ \to \bb T^d_N$ which are right-continuous and have left-limits.
Denote by $\mb P_x$, $x\in \bb T^d_N$, the probability measure on
$D(\bb R_+, \bb T^d_N)$ induced by the Markov chain $\eta^N(t)$
starting from $x$. Expectation with respect to $\mb P_x$ is denoted by
$\mb E_x$.

Denote by $B(x,\ell)$, $x\in\bb T^d_N$, $\ell\ge 1$, the ball centered
at $x$ and of radius $\ell$ in the graph $\bb T^d_N$. For a subset $A$
of $\bb T^d_N$, denote by $H_A$ (resp. $H^+_A$) the hitting time of
(resp. the return time to) the set $A$:
\begin{equation*}
\begin{split}
& H_A \;=\; \inf \big \{t>0 : \eta^N(t) \in A \big\}\;, \\
& \quad H^+_A \;=\; \inf \big \{t>0 : \eta^N(t) \in A
\text{ and } \exists\; 0<s<t \,;\, \eta^N(s)\not = \eta^N(0) \big\}\;.    
\end{split}
\end{equation*}
Let
\begin{equation*}
\ell_N \;=\; \frac{N}{(\log N)^{1/4}} \;\; \text{ if $d = 2\;,$ } \quad
\ell_N \;=\;  \sqrt{N} \;\; \text{ if $d \geq 3\;,$} 
\end{equation*}
and denote by $v_N$ the escape probability from a ball of
radius $\ell_N$:
\begin{equation}
\label{e29}
v_N \;=\; \mb P_x [ H_{B(x,\ell_N)^c} < H^+_x]\;.
\end{equation}

Denote by $\mf v_d$, $d\ge 3$, the probability that a
nearest-neighbor, symmetric random walk on $\bb Z^d$ never returns to
its initial state, an let
\begin{equation*}
\theta_N \;=\;
\frac 2\pi\, \log N \;\; \text{ if $d = 2\;$, } \quad 
\theta_N \;=\; \frac 1{\mf v_d} \;\; \text{ if $d \geq 3\;$.}
\end{equation*}
By \cite[Theorem~1.6.6]{Law91}, $\lim_{N\to\infty} \theta_N\, v_N
=1$. 

We investigated in \cite{jlt1, jlt2} the asymptotic behavior of the
rescaled process 
\begin{equation}
\label{e30}
\bb X_N(t) \;=\; X_N(t\theta_N)\;.
\end{equation}
Actually, we examined this problem for a large class of random walks
evolving on random graphs with i.i.d. random weights $W_j$ in the basin
of attraction of an $\alpha$-stable distribution, $0<\alpha<1$.

That the convergence of the rescaled process $\bb X_N(t)$ can not
occur in the Skorohod topology is easy to understand. Consider the
two-dimensional case, for instance. In the time scale $\theta_N$, a
typical trajectory which starts from a site $j$ has the shape
illustrated in Figure \ref{efig1} with the difference that the number
of sites visited in time intervals $R_j$ is much greater than the one
depicted in Figure \ref{efig1}.

\begin{figure}[htb]
  \centering
  \def\svgwidth{300pt}
  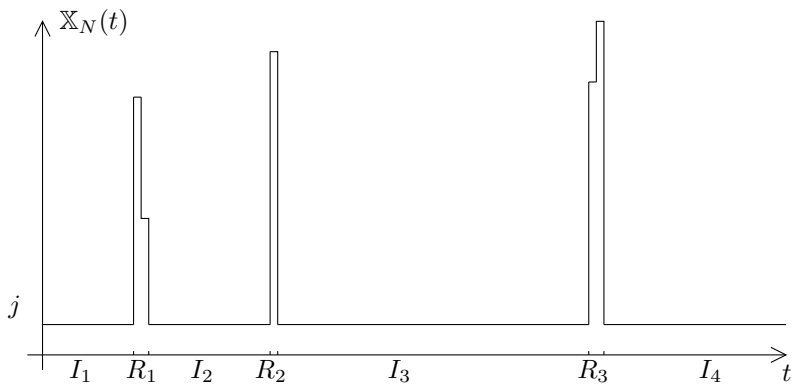
\caption{A typical trajectory of a random walk among random traps in
  $\bb T^d_N$.}
\label{efig1}
\end{figure}

The time intervals $I_1$, $I_2, \dots$ correspond to the holding times
at $x^N_j$ of the rescaled process $\eta^N(t \theta_N)$, while the time
intervals $R_1$, $R_2, \dots$ correspond to the excursions in the ball
$B(x^N_j, \ell_N)$ between two visits to $x^N_j$. In the time scale
$\theta_N$, the sojourn times at $x^N_j$ have length of order
$\theta_N^{-1}$, while we proved in \cite{jlt1} that the length of the
time intervals $R_j$ are of an order $\gamma_N$ which is much smaller
than $\theta^{-1}_N$, $\gamma_N\ll \theta^{-1}_N$. This property follows
from the fact that with a very large probability, excluding $x^N_j$,
all sites in the ball $B(x^N_j, \ell_N)$ are shallow traps, i.e.,
belong to the set $\{x^N_1, \dots, x^N_{M_N}\}^c$ for some sequence
$M_N$ which increases to $\infty$ fast enough.

The process $\eta^N(t \theta_N)$ reaches the boundary of the ball
$B(x^N_j, \ell_N)$ after $v_N^{-1} \sim \theta_N$ attempts.  Summing
the lengths of the time intervals $I_j$ up to the hitting time of the
set $B(x^N_j, \ell_N)^c$, we obtain a total length of order $1$.
Actually, the total length is a mean $\theta_N$ geometric sum of
i.i.d. mean $W_j \theta^{-1}_N$ exponential random variables and is
therefore distributed according to a mean $W_j$ exponential random
variable, which foretells a Markovian limit.  In contrast, the sum of
the lengths of the time intervals $R_j$ is negligible, being of order
$\gamma_N \theta_N\ll 1$.

As in the zero-range model, the short excursions in the ball $B(x^N_j,
\ell_N)$ prevent the convergence in the Skorohod topology of the
trajectory presented in Figure \ref{efig1} to the constant trajectory
equal to $j$.

\smallskip\noindent{\bf 3. Trace, last visit and soft topology.}  There
are three ways to overcome the topological obstruction caused by the short
excursions, illustrated in Figures \ref{efig2} and \ref{efig1} by the
time intervals $R_j$, $j\ge 1$. The first one, proposed in \cite{bl2,
  bl7}, consists in removing the time intervals $R_j$ by considering
the \emph{trace} of the process $\bb X_N(t)$ on the set $A_N$, where
$A_N$ represents the set $S_L=\{1, \dots, L\}$ in the zero-range
model, and the set of deep traps $\{1, \dots, M_N\}$ in the second
model.

Denote by $\bb X^{\rm T}_N(t)$ the trace of $\bb X_N(t)$ on the set
$A_N$. In the random walk model, since the sites in the ball $B(x^N_j,
\ell_N)$, with the exception of $x^N_j$, belong to $A_N^c$, the trace
on the set $A_N$ of the trajectory presented in Figure \ref{efig1} is
the constant trajectory equal to $j$. Therefore, by considering the
trace process $\bb X^{\rm T}_N(t)$, one removes the short excursions
among the shallow traps, and one is left with the problem of proving
that the trace process $\bb X^{\rm T}_N(t)$ converges in the Skorohod
topology to some Markov chain. The same ideas apply to the zero-range
model. In fact, this strategy has been adopted in \cite{bl3, l1} to
describe the asymptotic evolution of the condensate in sticky
zero-range dynamics.

A second way to overcome the difficulty created by the short time
intervals $R_j$ is to consider the process which records the last
visit to the set $A_N$. Denote by $\bb X^{\rm V}_N(t)$ the process
which at time $t$ is equal to last site in $A_N$ visited by the
process $\bb X_N(t)$ before time $t$. More precisely, let
\begin{equation*}
\sigma_N(t) \;:=\; \sup \big \{s\le t : \bb X_N(s) \in A_N \big\}\;, 
\end{equation*}
with the convention that $\sigma_N(t) = 0$ if the set $\{s\le t : \bb
X_N(s) \in A_N \}$ is empty. Assume that $\bb X_N(0)$ belongs to $A_N$
and define $\bb X^{\rm V}_N(t)$ by
\begin{equation*}
\bb X^{\rm V}_N(t) \;=\; 
\begin{cases}
\bb X_N (\sigma_N(t)) & \text{if $\;\bb X_N (\sigma_N(t))\in A_N$,} \\
\bb X_N (\sigma_N (t)-) & \text{if $\;\bb X_N(\sigma_N(t))\not\in A_N$.} 
\end{cases}
\end{equation*}

We refer to $\bb X^{\rm V}_N(t)$ as the \emph{last visit process}. In
Figure \ref{efig1}, assuming that all sites in $B(x^N_j, \ell_N)$,
except $x^N_j$, are shallow traps, the last visit trajectory is
constant equal to $j$, the short excursions among the shallow traps in
the time intervals $R_i$ being replaced by trajectories which are
constant equal to $j$. As for the trace process, after this
surgical intervention on the paths, one is left to prove that the
last visit process $\bb X^{\rm V}_N(t)$ converges in the Skorohod
topology to some Markovian dynamics.

Proposition 4.3 in \cite{bl2} asserts that if the trace on a set $A_N$
of an $\Omega_N$-valued chain $\bb X_N(t)$ converges in the Skorohod
topology to a Markov process $\bb X(t)$, and if the time spent by $\bb
X_N(t)$ on the complement $\Omega_N\setminus A_N$ is negligible, then
the process which records the last visit to the set $A_N$ also
converges in the Skorohod topology to the Markov process $\bb X(t)$.

The last visit process $\bb X^{\rm V}_N(t)$ has the advantage with
respect to the trace process that it does not translate in time the
original trajectory. More precisely, if $\bb X_N(t)$ belongs to $A_N$
then $\bb X^{\rm V}_N(t) = \bb X_N(t)$, while this may be false for
the trace process $\bb X^{\rm T}_N(t)$ because by removing short time
intervals, the value of the trace process at time $s$ corresponds to
the value of the original process at some later time $s'\ge s$: $\bb
X^{\rm T}_N(s) = \bb X_N(s')$ for some $s'\ge s$.

A third way to overcome the problem created by the short time
intervals $R_j$ is to define a topology, weaker than the Skorohod
topology, which disregards the behavior of the trajectory in short
time intervals. A first attempt in this direction has been made in
\cite{jlt2}, where we introduced a metric in the space of functions
$\mf x:[0,T]\to \bb R$ which induces the topology of the convergence
in measure. We proved in \cite{jlt2} that the rescaled process $\bb
X_N(t)$ introduced in \eqref{e30} converges in this metric to the
$K$-process.

In this article, we introduce another topology in which we can prove
the convergence to a Markov chain of the two models introduced above
and of all the other dynamics in which a metastable behavior has been
identified \cite{bl3, jlt1, bl4, jlt2, clmst, cmt, bl7, lt1}.

This topology has two advantages with respect to the the one
introduced in \cite{jlt2}. On the one hand, it is defined on the space
of paths which are \emph{soft} right-continuous and have soft
left-limits, a much smaller space than the one which appears in
\cite{jlt2}. Actually, this set of paths, denoted by $E([0,T], S)$ in
the next section, is precisely the set of trajectories which supports
the paths of Markov chains which have instantaneous states
\cite[Chapter II.7]{C1}, \cite[Section 9.2]{F1}. On the other hand,
this topology is a natural generalization of the Skorohod topology and
is connected to the Skorohod topology through the last visit process
(cf. Theorem \ref{se14}).

The main contributions of this article are the introduction of the
soft topology, defined by the metric $\mb d$ in \eqref{e06}, and two
results. The first one, Theorem \ref{se14}, establishes that a
sequence of probability measures $P_n$ defined on $E([0,T], S_{\mf
  d})$ converges in the soft topology to a probability measure $P$ if
and only if for every $m\ge 1$, the sequence of probability measures
$P_n\circ \mf R^{-1}_m$ defined on $D([0,T], S_m)$ converges in the
Skorohod topology to the probability measure $P\circ \mf
R^{-1}_m$. Here $S_{\mf d}$ stands for the one point compactification
of $\bb N$, $S_m$ for the set $\{1, \dots, m\}$, and $\mf R_m x$ for
the path which records the last site in $S_m$ visited by the
trajectory $x\in E([0,T], S_{\mf d})$. The second result, Theorem
\ref{se12}, presents sufficient conditions on a sequence of
probability measures $P_n$, defined on $E([0,T], S_{\mf d})$ and
converging to a probability measure $P$ in the soft topology, for the
limit $P$ to be concentrated on a subspace of $D([0,T], S_{\mf d})$,
the space of c\`adl\`ag trajectories.

\smallskip
\noindent{\bf 4. Scaling limit of metastable Markov chains.}
In view of these results, to prove that a sequence of Markov chains
converges in the soft topology to a Markov chain evolving in a
contracted state space, we may proceed as follows. We first introduce
a general framework.

Consider a sequence of countable state spaces $E_N$, $N\ge 1$, and a
sequence of $E_N$-valued continuous-time Markov chains
$\eta^N(t)$. Denote by $\mb P_\eta$, $\eta\in E_N$, the probability
measure on the path space $D(\bb R_+, E_N)$ induced by the Markov
chain $\eta^N(t)$ starting from $\eta$. Expectation with respect to
$\mb P_\eta$ is denoted by $\mb E_\eta$.

Let $\ms E^1_N, \dots, \ms E^L_N$, $L\ge 2$, be a finite number of
disjoint subsets of $E_N$: $\ms E^x_N\cap \ms E^y_N=\varnothing$,
$x\neq y$. The sets $\ms E^x_N$ have to be interpreted as wells of the
Markov chains $\eta^N(t)$.  Let $S_L = \{1, \dots, L\}$, $\ms
E_N=\cup_{x\in S_L}\ms E^x_N$ and $\Delta_N=E_N \setminus \ms E_N$ so
that
\begin{equation}
\label{e17}
\{\ms E^1_N \,,\ \dots \,,\, \ms E^{L}_N \,,\, \Delta_N\} \text{ forms
a partition of } E_N\;. 
\end{equation}
We assumed here that the number of wells, $L$, does not depend on $N$,
but the same analysis can be carried through if it depends on $N$, as in
the case of random walks among random traps.

Denote by $\eta^{\rm T}(t)$ the trace of the process $\eta^N(t)$ on
the set $\ms E_N$, and by $\eta^{\rm V} (t)= (\mf R_{\ms E}
\eta^N)(t)$ the process which records the last site visited by $\eta^N
(t)$ in the set $\ms E_N$, as defined in \eqref{e05}.  Denote by $\Psi
= \Psi_N:\ms E_N\mapsto S_L \cup \{N\}$, the projection given by
\begin{equation*}
\Psi(\eta) \;=\; \sum_{x=1}^L x \, 
\mathbf 1\{\eta \in \ms E^x_N\} \;+\; 
N \mathbf 1\{\eta \in \Delta_N\}\;.
\end{equation*}
We call $\Psi(\eta)$ the order parameter.  Let $\{X_N(t): t\ge 0\}$
(resp. $X^{\rm T}_N(t)$, $X^{\rm V}_N(t)$) be the stochastic process
on $S_L \cup\{N\}$ (resp. $S_L$) defined by $X_N(t)=\Psi(\eta^{N}(t))$
(resp. $X^{\rm T}_N(t)=\Psi(\eta^{\rm T}(t))$, $X^{\rm
  V}_N(t)=\Psi(\eta^{\rm V} (t))$). Besides trivial cases, $X_N(t)$ is
not Markovian.  Note that $X^{\rm T}_N(t)$ is the trace of $X_N(t)$ on
the set $S_L$ and that $ X^{\rm V}_N(t)$ is the process which records
the last site in $S_L$ visited by $X_N(t)$.

\begin{definition}
\label{se01}
Let $\nu_N$ be a sequence of probability measures on $\ms E_N$ such
that $\nu_N \circ \Psi^{-1}$ converges to a probability measure $\nu$
on $S_L$.  The sequence of Markov chains $\{\eta^N(t) : t\ge 0\}$ is
said to be a metastable sequence of Markov chains for the partition
\eqref{e17} and the initial state $\nu_N$ if there exist
\begin{enumerate}
\item[{\rm (a)}] An increasing sequence $\theta_N$, $\theta_N\gg 1$,
\item[{\rm (b)}] A $S_L$-valued Markov chain $\bb X(t)$ whose
  distribution we denote by $\bb P_x$, $x\in S_L$,
\end{enumerate}
such that the measure $\mb P_{\nu_N} \circ \bb X_N^{-1}$, $\bb X_N(t)
= X_N(t \theta_N) =\Psi(\eta^{N}(t\theta_N))$, converges in the soft
topology to $\bb P_\nu = \sum_{x\in S_L} \nu(x) \bb P_x$.
\end{definition}

To prove that a sequence of Markov chain is metastable one may proceed
as follows:

\smallskip\noindent{\it Step 1:} Prove the convergence in the Skorohod
topology of $\bb X^{\rm T}_N(t) = X^{\rm T}_N(\theta_N t)$ to
a Markov chain $\bb X(t)$.

\smallskip\noindent{\it Step 2:} Prove that the time spent by the
chain $\eta^N(t)$ on the set $\Delta_N$ in the time scale $\theta_N$
is negligible.

\smallskip\noindent{\it Step 3:} Apply Theorem \ref{se15} which
asserts that under the two previous conditions the process $\bb
X_N(t)$ converges in the soft topology to $\bb X(t)$. \smallskip

This article is organized as follows. For a metric space $\bb M$,
denote by $D([0,T], \bb M)$, $T>0$, the space of right-continuous
functions $x:[0,T]\to\bb M$ with left-limits. We introduce in
\eqref{e06} a metric $\mb d$ in a subspace of $D([0,T], S_{\mf d})$,
where $S_{\mf d}$ is the one-point compactification of $\bb N$. The
completion of this subspace with respect to the metric $\mb d$ turns
out to be the space of soft right-continuous trajectories with soft
left-limits.  This space, denoted by $E([0,T], S_{\mf d})$, is
introduced in Section \ref{esec5}. We derive in this section several
properties of the metric $\mb d$ and we prove in Proposition
\ref{se06} that the space $E([0,T], S_{\mf d})$ endowed with the
metric $\mb d$ is complete and separable. In Section \ref{esec2} we
introduce a subspace $E^*([0,T], S_{\mf d})$ of the space $E([0,T],
S_{\mf d})$ and we present in Lemma \ref{se11} sufficient conditions
for the limit trajectory $x$ in the soft topology of a sequence $x_n$
in $E([0,T], S_{\mf d})$ to belong to $E^*([0,T], S_{\mf d})$. In
Section \ref{esec3} we state and prove the main results of the
article, mentioned above, and in Section \ref{esec4} we present some
applications of these results. We prove the convergence of the order
parameter to a Markov chain in the case of the random walk among
random traps presented above and in the case of the condensate in
sticky zero-range dynamics.


\section{The space $E([0,T], S_{\mf d})$}
\label{esec5}


Let $S_m = \{1, \dots, m\}$, $m\ge 1$, and let $S_{\mf d}$ be the
one-point compactification of $S = \bb N$: $S_{\mf d} = S \cup\{\mf
d\}$, $\mf d= \infty$, where the metric in $S_{\mf d}$ is given by
$d(k,j)=|k^{-1} - j^{-1}|$. 


Fix $T>0$. Denote by $\Lambda$ the set of increasing and continuous
functions $\lambda:[0,T]\to [0,T]$ such that $\lambda(0)=0$, $\lambda
(T)=T$.  For $\lambda\in \Lambda$, let
\begin{equation*}
\Vert\lambda\Vert^o \;=\; \sup_{0\le s<t\le T} 
\Big | \, \log \frac{\lambda (t) - \lambda (s)}{t-s} \, \Big | \;.
\end{equation*}
Denote by $d_S$ the Skorohod metric on $D([0,T], S_m)$, $m\ge 1$,
defined by
\begin{equation*}
d_S (x,y) \;=\; \inf_{\lambda\in \Lambda} 
\max \Big\{  \Vert x - y \lambda \Vert_\infty \,,\,
\Vert\lambda\Vert^o \Big\}\;,
\end{equation*}
where $y\lambda = y\circ \lambda$ and $\Vert x - y \lambda
\Vert_\infty = \sup_{0\le t\le T} d ( x (t), y \lambda (t))$.

\begin{definition}
\label{se02}
A measurable function $x: [0,T] \to S_{\mf d}$ is said to have a
\emph{soft} left-limit at $t\in (0,T]$ if one of the following two
alternatives holds
\begin{enumerate}

\item[{\rm (a)}] The trajectory $x$ has a left-limit at $t$, denoted
  by $x(t-)$; 

\item[{\rm (b)}] The set of cluster points of $x(s)$, $s\uparrow t$,
  is a pair formed by $\mf d$ and a point in $S$, denoted by $x(t\,
  \ominus)$. 

\end{enumerate}

\noindent
A \emph{soft} right-limit at $t\in [0,T)$ is defined analogously. In
this case, the right-limit, when it exists, is denoted by $x(t+)$, and
the cluster point of the sequence $x(s)$, $s\downarrow t$, which
belongs to $S$ when the second alternative is in force is denoted by
$x(t\, \oplus)$.
\end{definition}


More concisely, a trajectory $x$ has a soft left-limit at $t\in (0,T]$
if and only if there exists $n\in S$ such that for all $m\ge 1$, there
exists $\delta>0$ for which $x(s)\in \{n\}\cup S^c_{m}$ for all
$t-\delta<s<t$. Note the similitude of this definition with the notion
of quasiconvergence in \cite{F1}.

The second alternative in Definition \ref{se02} asserts that there
exist $n\in S$ and two increasing sequences $t_j$, $t'_j \uparrow t$
such that $\lim_j x(t_j) = n$, $\lim_j x(t'_j) = \mf d$. Moreover, if
$x(t''_j)$ converges for some sequence $t''_j \uparrow t$, $\lim_j
x(t''_j) \in \{n, \mf d\}$.

We call $x(t\, \ominus)$ the \emph{finite} soft left-limit of $x$ at
$t$.  Whenever we refer to $x(t-)$ it means that $x$ has a left-limit
at $t$. Similarly, when we refer to $x(t\, \ominus)$, it is understood
that $x$ has not a left-limit at $t$, but that the alternative (b) of
the previous definition is in force. An analogous convention is
adopted for $x(t+)$ and $x(t\,\oplus)$.

\begin{remark} 
\label{se03}
Since $S_{\mf d}$ is a compact set, to prove that $x$ has a soft
right-limit at $t$ we only have to show uniqueness of limit points in
$S$ (assuming they exist). In other words, we have to prove that if
$t_j$ and $t'_j$ are sequences decreasing to $t$ and if $x(t_j)$,
$x(t'_j)$ converge to $m\in S$, $n\in S$, respectively, then $m=n$.
\end{remark}

\begin{definition} 
\label{se04}
A trajectory $x:[0,T] \to S_{\mf d}$ which has a soft right-limit at
$t$ is said to be soft right-continuous at $t$ if one of the following
three alternatives holds

\begin{enumerate}
\item[{\rm (a)}] $x(t+)$ exists and is equal to $\mf d$;
\item[{\rm (b)}] $x(t+)$ exists, belongs to $S$, and $x(t+) = x(t)$;
\item[{\rm (c)}] $x(t\oplus)$ exists and $x(t\oplus) = x(t)$.
\end{enumerate}

\noindent A trajectory $x: [0,T] \to S_{\mf d}$ which is soft
right-continuous at every point $t\in [0,T]$ is said to be soft
right-continuous.
\end{definition}

A trajectory $x$ is soft right-continuous at $t$ if and only if there
exists $n\in S$ such that for all $m\ge 1$, there exists $\delta>0$
for which $x(s)\in \{n\}\cup S^c_{m}$ for all $t \le s<t+\delta$.

Note that if $x$ is soft right-continuous at $t$ and if $x(t+)=\mf d$,
then $x(t)$ may be different from $x(t+)$. In contrast, if $x$ is soft
right-continuous at $t$ and if $x(t)=\mf d$, then $x(t+)$ exists and
$x(t+) = \mf d = x(t)$.

Clearly, if $x$ is soft right-continuous at $t$, for every $m\ge 1$,
there exists $\epsilon>0$ such that for all $t\le s<t+\epsilon$,
\begin{equation}
\label{e08}
x(s) = x(t) \text{ or } x(s)\ge m\;.
\end{equation}
Similarly, if $x$ has a soft left-limit at $t$, there exists $n\in S$
with the following property.  For every $m\ge 1$, there exists
$\epsilon>0$ such that for all $t-\epsilon< s<t$,
\begin{equation}
\label{e09}
x(s)= n \text{ or } x(s)\ge m\;.
\end{equation}

\begin{definition}
Let $\bb E([0,T], S_{\mf d})$ be the space of soft right-continuous
trajectories $x: [0,T] \to S_{\mf d}$ with soft left-limits.
\end{definition}
 
Note that the space $\bb E([0,T], S_{\mf d})$ corresponds to the space
of sample functions of Markov chains with instantaneous states
\cite[Chapter II.7]{C1}, \cite[Section 9.2]{F1}.


Fix a trajectory $x$ in $\bb E([0,T], S_{\mf d})$ such that $x(r)=\mf d$
for some $r\in [0,T)$. Since $x$ is soft right-continuous, by
Definition \ref{se04}, 
\begin{equation}
\label{e15}
\text{$x(r+)$ exists and $x(r+)=\mf d$.}
\end{equation}

Recall that $S_m = \{1, \dots, m\}$, $m\ge 1$. For a trajectory $x$ in
$\bb E([0,T], S_{\mf d})$, $t\in [0,T]$, let
\begin{equation}
\label{e07}
\sigma^x_m(t) \;:=\; \sup\{s\le t : x(s) \in S_m\}\;.
\end{equation}
If the set $\{s\le t : x(s) \in S_m\}$ is empty, we set $\sigma^x_m(t)
= 0$, but this convention does not play any role below and we could
have defined $\sigma^x_m(t)$ in another way. When there is no
ambiguity and it is clear to which trajectory we refer to, we denote
$\sigma^x_m(t)$ by $\sigma_m(t)$.


\begin{figure}[htb]
  \centering
  \def\svgwidth{350pt}
  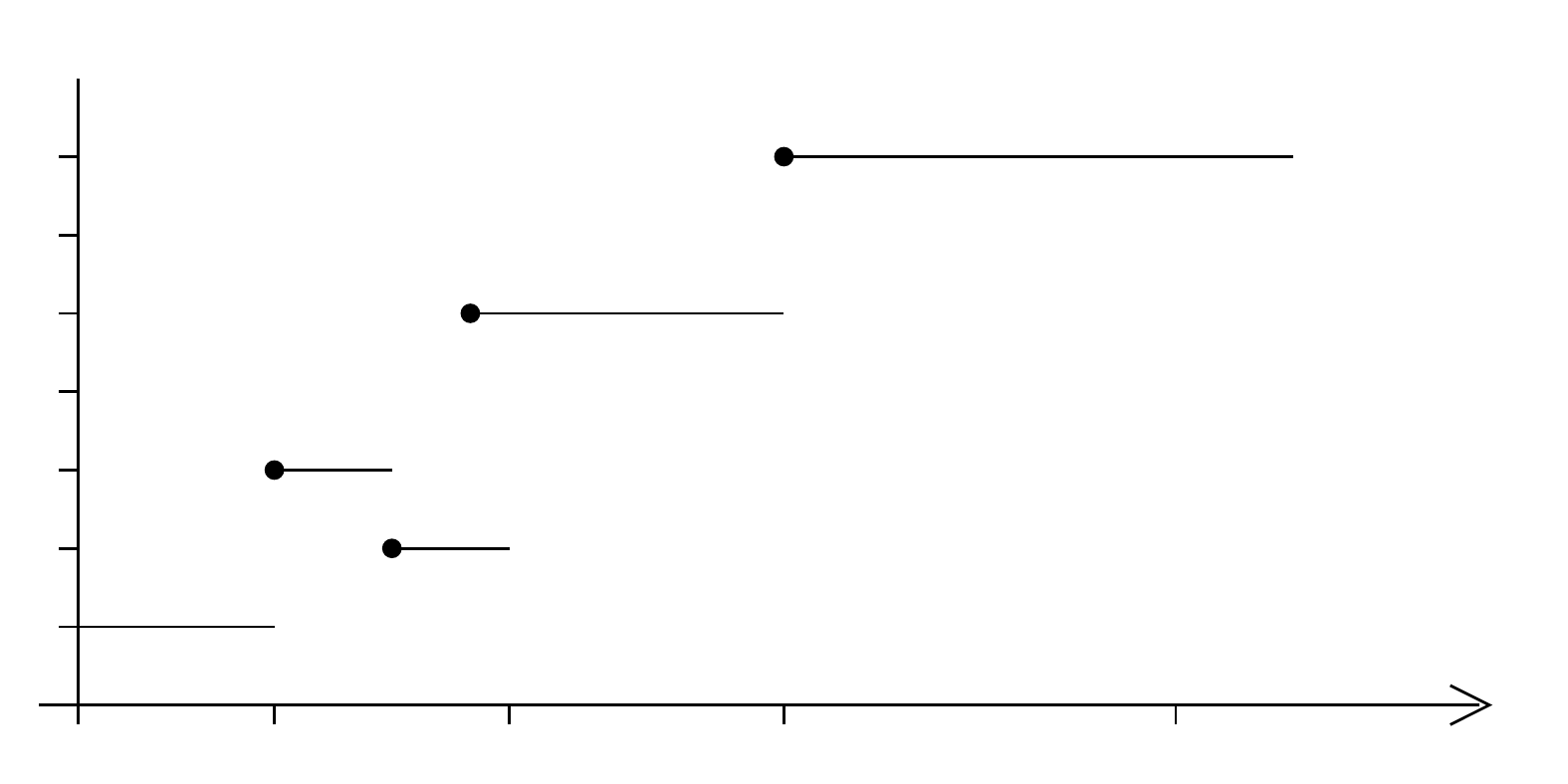
  \caption{The values of $\sigma_j(t)$ and $(\mf R_j x)(t)$ for a
    trajectory $x:[0,T] \to \bb N$. In this example $\sigma_2(t) =
    \sigma_3(t) = \sigma_4(t)$, $\sigma_5(t) = \sigma_6(t)$, and $(\mf
    R_2 x)(t) = (\mf R_3 x)(t) = (\mf R_4 x)(t) = 2$, $(\mf R_5 x)(t)
    = (\mf R_6 x)(t) = 5$.  }
\label{efig4}
\end{figure}

Fix $t\in (0,T]$ and $m\ge 1$. Suppose that $\sigma_m(t)>0$ and that
$x(\sigma_m(t))\not\in S_m$, so that $x(s)\not \in S_m$ for
$\sigma_m(t) \le s\le t$. By \eqref{e09}, there exist $n\in S$ and
$\epsilon>0$ such that for each $s\in (\sigma_m(t) - \epsilon ,
\sigma_m(t))$ either $x(s)=n$ or $x(s)>m$. By definition of
$\sigma_m(t)$ we must have $n\in S_m$. Moreover, $x(\sigma_m(t)-)=n$
if $x$ has a left-limit at $\sigma_m(t)$, and
$x(\sigma_m(t)\ominus)=n$ if not.


Let ${\mf R}_m x$ be the trajectory which records the last site
visited in $S_m$: $({\mf R}_m x)(t) = 1$ if $x(s)\not \in S_m$
for $0\le s\le t$, and
\begin{equation}
\label{e05}
({\mf R}_m x)(t) \;=\; 
\begin{cases}
x(\sigma_m(t)) & \text{if $x(\sigma_m(t))\in S_m$,}\\
x(\sigma_m(t)-) & \text{if $x(\sigma_m(t))\not \in S_m$ and 
$x(\sigma_m(t)-)$ exists,} \\
x(\sigma_m(t)\ominus) & \text{otherwise},
\end{cases}
\end{equation}
if there exists $0\le s\le t$ such that $x(s)\in S_m$.  Figure
\ref{efig4} illustrates the definition of $\sigma_m(t)$ and ${\mf R}_m
x$ for some trajectory $x$.

Note that $({\mf R}_m x)(0) = x(0)$ if $x(0)\in S_m$ and $({\mf R}_m
x)(0) = 1$ if $x(0)\not\in S_m$. The convention that $({\mf R}_m x)(t)
= 1$ if $x(s)\not \in S_m$ for $0\le s\le t$ corresponds to the
assumption that the trajectory $x$ is defined for $t<0$ and that
$x(t)=1$ in this time interval.

Consider a trajectory $x$ in $D([0,T],S_{\mf d})$, $m\ge 1$ and $t\in
(0,T]$. Assume that $x(t)\not\in S_m$ and that there exists $0\le s\le
t$ such that $x(s)\in S_m$. Since $x$ is right-continuous,
$\sigma_m(t)>0$ and $x(\sigma_m(t)) = x(\sigma_m(t)+) \not \in
S_m$. Hence, since $x$ has left-limits, under the above conditions,
\begin{equation}
\label{e23}
({\mf R}_m x)(t) \;=\; x(\sigma_m(t)-)\;. 
\end{equation}
Note that we may have $\sigma_m(t) = t$ in this example.

\begin{asser}
\label{ea01}
Fix a trajectory $x$ in $\bb E([0,T], S_{\mf d})$.  For each $m\ge 1$,
${\mf R}_m x$ is a trajectory in $D([0,T], S_m)$.
\end{asser}

\begin{proof}
Fix $m\ge 1$.  We first prove the right continuity of ${\mf R}_m
x$. Fix $t\in [0,T)$.  By \eqref{e08}, there exists $\delta>0$ such
that for all $t\le s\le t+\delta$, either $x(s) = x(t)$ or $x(s)
>m$. Suppose that $x(t)$ belongs to $S_m$. In this case, $({\mf R}_m
x)(s) = x(t) = ({\mf R}_m x)(t)$ for $t\le s < t+\delta$. On the other
hand, if $x(t) \not \in S_m$, $x(s)\not\in S_m$ for $t\le s<t+\delta$
so that $\sigma_m(s) = \sigma_m(t)$ in this interval. Therefore, in
view of \eqref{e05}, $({\mf R}_m x)(s) = ({\mf R}_m x)(t)$ for $t\le
s\le t+\delta$.  This proves that ${\mf R}_m x$ is right-continuous.

We turn to the proof of the existence of a left limit at $t\in
(0,T]$. If $x(t-)$ exists and belongs to $S_m$, $({\mf R}_m x)(s) =
x(t-)$ for all $s<t$ close enough to $t$. If $x(t-)$ exists and does
not belong to $S_m$, $\sigma_m(s)$ is constant in an open interval
$(t-\delta, t)$, which implies that $({\mf R}_m x)(s)$ is constant in
the same interval. Finally, suppose that $x(t\,\ominus)$ exists. In
view of \eqref{e09}, there exists $\delta>0$ such that for all
$t-\delta<s< t$, either $x(s)>m$ or $x(s)=x(t\,\ominus)$. If
$x(t\,\ominus)\le m$, $({\mf R}_m x)(s) = x(t\,\ominus)$ in some
interval $(t-\delta',t)$, $\delta'>0$. If $x(t\,\ominus)>m$, then
$\sigma_m(s)$ is constant in the interval $(t-\delta, t)$, so that
${\mf R}_m x$ is constant in the same interval. This concludes the
proof of the assertion. 
\end{proof}

The next example shows that the trajectories ${\mf R}_m x$, $m\ge 1$,
do not characterize the trajectory $x$.

\begin{example}
\label{se08}
Fix $0<s<t<T$ and a sequence $\{t_j : j\ge 1\}$ such that $t_1<T$,
$t_j\downarrow t$. Consider the trajectories $x$, $y\in
\bb E(([0,T], S_{\mf d})$ given by
\begin{equation*}
\begin{split}
& x \;=\; \mb 1\{[0,s)\} \;+\; \mf d \, \mb 1\{[s,t]\} \;+\;
\sum_{j\ge 2} j\, \mb 1\{[t_j,t_{j-1})\} \;+\;
\mb 1\{[t_1,T]\}\;,\\
& \quad y \;=\; \mb 1\{[0,t]\} \;+\;
\sum_{j\ge 2} j\, \mb 1\{[t_j,t_{j-1})\} \;+\; \mb 1\{[t_1,T]\}\;.
\end{split}
\end{equation*}
It is clear that ${\mf R}_m x = {\mf R}_m y$ for all $m\ge 1$.
\end{example}


For a trajectory $x\in \bb E([0,T], S_{\mf d})$, let
$\sigma^x_\infty(t)$ be the time of the last visit to $S$:
\begin{equation*}
\sigma^x_\infty(t) \;:=\; \sup\{s\le t : x(s) \in S\}\;,
\end{equation*}
with the convention that $\sigma^x_\infty(t)=0$ if $x(s)=\mf d$ for
$0\le s\le t$. As before, when there is no ambiguity and it is clear
to which trajectory we refer to, we denote $\sigma^x_\infty(t)$ by
$\sigma_\infty(t)$.


Let ${\mf R}_\infty x$ be the trajectory which records the last site
visited in $S$: $({\mf R}_\infty x)(t) = 1$ if $x(s) = \mf d$ for all
$0\le s\le t$, and
\begin{equation*}
({\mf R}_\infty x)(t) \;=\; 
\begin{cases}
x(\sigma_\infty(t)) & \text{if $x(\sigma_\infty(t))\in S$,} \\
x(\sigma_\infty(t)-) & \text{if $x(\sigma_\infty(t))\not\in S$ 
and if $x(\sigma_\infty(t)-)$ exists,} \\
x(\sigma_\infty(t)\ominus) & \text{otherwise},
\end{cases}
\end{equation*}
if there exists $0\le s \le t$ such that $x(s)\in S$. As for the
operator $\mf R_m$, the convention that $({\mf R}_\infty x)(0) = 1$ if
$x(0) = \mf d$ corresponds in assuming that the trajectory is defined
for $t<0$ and that $x(t)=1$ for $t<0$.  Note that $(\mf R_\infty x)(0)
\in S$ and that $(\mf R_\infty x)(0)= x(0)$ if and only if $x(0)\in
S$. Note also that in Example \ref{se08} $y=\mf R_\infty x$ and that
$y$ is not right-continuous at $t$ but soft right-continuous.

Consider a trajectory $x$ in $D([0,T],S_{\mf d})$ and $t\in
(0,T]$. Assume that $x(t)\not\in S$ and that there exists $0\le s\le
t$ such that $x(s)\in S$. Since $x$ is right-continuous,
$\sigma_\infty(t) >0$ and $x(\sigma_\infty(t)) = x(\sigma_\infty (t)+)
\not \in S$. Hence, since $x$ has left-limits, under the above
conditions,
\begin{equation}
\label{e24}
({\mf R}_\infty x)(t) \;=\; x(\sigma_\infty (t)-)\;. 
\end{equation}


\begin{definition}
Denote by $E([0,T], S_{\mf d})$ the set of trajectories in $\bb
E([0,T], S_{\mf d})$ such that $x(0)\in S$ and which fulfill the
following condition.  If $x(t)=\mf d$ for some $t\in (0,T]$, then
$\sigma_\infty (t)>0$ and $x(\sigma_\infty(t)) =
x(\sigma_\infty(t)-)=\mf d$.
\end{definition}

In words, a trajectory $x$ belongs to $E([0,T], S_{\mf d})$ if it
possesses the following property. If $x(t)=\mf d$ for some $t\in
(0,T]$, then $x$ has visited $S$ before time $t$, $\sigma_\infty
(t)>0$, and at the time of the last visit to $S$ before $t$, that is,
at time $\sigma_\infty(t)$, $x$ is equal to $\mf d$ and its left-limit
is also equal to $\mf d$: $x(\sigma_\infty(t)) =
x(\sigma_\infty(t)-)=\mf d$.  Note that the trajectory $x$ of Example
\ref{se08} does not belong to $E([0,T], S_{\mf d})$ because
$x(\sigma_\infty(t)-)=1$.

\begin{lemma} 
\label{se10}
For every $x\in \bb E([0,T], S_{\mf d})$, the trajectory ${\mf
  R}_\infty x$ belongs to the space $E([0,T], S_{\mf d})$.
\end{lemma}

\begin{proof}
Fix a trajectory $x$ in $\bb E([0,T], S_{\mf d})$. By definition $(\mf
R_\infty x)(0) \in S$. We first show that
${\mf R}_\infty x$ belongs to $\bb E([0,T], S_{\mf d})$. 

We claim that ${\mf R}_\infty x$ has a left-limit at $t\in(0,T]$ if
$x$ has one. Suppose first that $x(t-)=\mf d$. If there exists
$\delta>0$ such that $x(s)=\mf d$ for $s\in (t-\delta, t)$, then
$\sigma_\infty$ is constant in this interval. By definition, ${\mf
  R}_\infty x$ is constant in the same interval and has therefore a
left-limit at $t$. On the other hand, if there exists a sequence
$t_j\uparrow t$ such that $x(t_j)\in S$, $\sigma_\infty (s) \ge t_1$
for $t_1\le s <t$. As $x(t-)=\mf d$, for every $m\ge 1$, there exists
$\delta>0$ such that $x(s)\ge m$ for $t-\delta\le s<t$. Therefore
$({\mf R}_\infty x)(s) \ge m$ for $t^*_\delta\le s <t$, where
$t^*_\delta$ is the smallest element of the sequence $t_j$ which is
greater than $t-\delta$.  This proves that $({\mf R}_\infty x)(t-)$
exists and is equal to $\mf d$. Suppose now that $x(t-)\in S$. In this
case $x(s)=x(t-)\in S$ for $s$ in some interval $(t-\delta, t)$. In
particular, $({\mf R}_\infty x) (s)= x(s) = x(t-)$ in the same
interval, which proves the claim. The trajectory $x$ of Example
\ref{se08} shows that the left-limits of $x$ and ${\mf R}_\infty x$ at
some point $t$ may be different.

Suppose now that $x(t\ominus)$ exists and is equal to $n\in S$.  By
definition there exists a sequence $t_j\uparrow t$ such that $x(t_j)
\to n$, which means that $x(t_j) = n$ for $j$ sufficiently large.  By
definition, $({\mf R}_\infty x) (t_j) = n$ for the same indices. Fix
$m>n$. By \eqref{e09}, there exists $\delta>0$ such that $x(s) =n$ or
$x(s)\ge m$ for all $t-\delta<s<t$. Hence, if we denote again by
$t^*_\delta$ the smallest element of the sequence $t_j$ which is
greater than $t-\delta$, for $t^*_\delta < s< t$, $({\mf R}_\infty x)
(s) = n$ or $({\mf R}_\infty x) (s) \ge m$. This proves that ${\mf
  R}_\infty x$ has a soft left-limit at $t$.

The trajectory ${\mf R}_\infty x$ is soft right-continuous. Fix $t\in
[0,T)$ and assume that $x(t) = \mf d$. If $x(s)=\mf d$ in some
interval $(t,t+\epsilon)$, $\sigma_\infty$ and ${\mf R}_\infty x$ are
constant on the interval $[t,t+\epsilon)$; while if there exists a
sequence $t_j\downarrow t$ such that $x(t_j)\in S$ for all $j$, $({\mf
  R}_\infty x)(t+) = \mf d$. In both cases, ${\mf R}_\infty x$ is soft
right-continuous at $t$.

Suppose now that $x(t)$ belongs to $S$ so that $({\mf R}_\infty
x)(t)=x(t)\in S$. By soft right-continuity of $x$ at $t$, for a fixed
$m\ge 1$, there exists $\delta>0$ such that $x(s) \in \{x(t)\}\cup
S^c_m$ for all $t\le s<t+\delta$. By definition of ${\mf R}_\infty x$
the same property holds for ${\mf R}_\infty x$, which proves its soft
right-continuity. 

We conclude the proof of the lemma showing that ${\mf R}_\infty x$
belongs to $E([0,T], S_{\mf d})$. Fix $t\in (0,T]$ such that $({\mf
  R}_\infty x)(t)=\mf d$. Denote by $\sigma_\infty(t)$, $\hat
\sigma_\infty(t)$ the time of the last visit to $S$ before time $t$ of
the trajectory $x$, ${\mf R}_\infty x$, respectively. Clearly
$x(t)=\mf d$, otherwise $({\mf R}_\infty x)(t) = x(t)\in S$. We also
have that $\sigma_\infty(t)>0$ because if $\sigma_\infty(t)=0$, $({\mf
  R}_\infty x)(t)$ would belong to $S$: $({\mf R}_\infty x)(t) =1$ by
definition if $x(s) = \mf d$ for $0\le s\le t$, and $({\mf R}_\infty
x)(t) = x(0)\in S$ if $x(s) = \mf d$ for $0< s\le t$. It follows from
the definition of ${\mf R}_\infty x$ and from the identity $({\mf
  R}_\infty x)(t)=\mf d$ that $x(\sigma_\infty(t)) =
x(\sigma_\infty(t)-) =\mf d$.

Since $x(s) = \mf d$ for $\sigma_\infty(t)<s\le t$, and since
$x(\sigma_\infty(t)) = x(\sigma_\infty(t)-) =\mf d$, we have that
$({\mf R}_\infty x)(s)=\mf d$, $\sigma_\infty(t)\le s\le t$, $\hat
\sigma_\infty(t) = \sigma_\infty(t)>0$, $({\mf R}_\infty
x)(\hat\sigma_\infty(t)-)=\mf d= ({\mf R}_\infty
x)(\hat\sigma_\infty(t))$.
\end{proof}

\begin{asser}
\label{ea10}
Let $x$ be a trajectory in $E([0,T], S_{\mf d})$. Then, ${\mf
  R}_\infty x = x$.
\end{asser}

The proof of this assertion is elementary. It follows from this claim
and from Lemma \ref{se10} that ${\mf R}_\infty : \bb E([0,T], S_{\mf
  d}) \to E([0,T], S_{\mf d})$ is a projection.  The next assertion
shows that ${\mf R}_m x$ converges pointwisely to $x$ if $x$ belongs
to $E([0,T], S_{\mf d})$.

\begin{asser}
\label{ea02}
Fix a trajectory $x$ in $\bb E([0,T], S_{\mf d})$. Then, ${\mf R}_m x$
converges pointwisely and $\lim_m {\mf R}_m x = {\mf R}_\infty x$.  
\end{asser}

\begin{proof} 
It is clear from the definition of ${\mf R}_m x$ that ${\mf R}_m x \le
{\mf R}_{m+1} x$. In particular, the pointwise limit always
exists. Fix $0\le t\le T$ and suppose initially that $x(t)\in S$.
In this case, for $m>x(t)$, $({\mf R}_m x)(t) = ({\mf R}_\infty x)
(t)$. 

Suppose from now on that $x(t) = \mf d$. If $x(s)=\mf d$ for $0\le
s\le t$, $({\mf R}_m x)(t) = 1 = ({\mf R}_\infty x) (t)$ for all $m\ge
1$, while if $x(0)\in S$ and $x(s)=\mf d$ for $0< s\le t$, $({\mf R}_m
x)(t) = x(0) = ({\mf R}_\infty x) (t)$ for all $m\ge x(0)$. We may
therefore assume that there exists $0<s<t$ such that $x(s)\in S$ so
that $\sigma_\infty(t) = \sigma^x_\infty(t)>0$.

If $x(\sigma_\infty(t)) \in S$, for $m>x(\sigma_\infty(t))$ we have
that $({\mf R}_m x)(t) = ({\mf R}_\infty x) (t)$, while if
$x(\sigma_\infty(t)) = \mf d$ and if $x(\sigma_\infty(t)-)$ exists,
$({\mf R}_\infty x) (t) = x(\sigma_\infty(t)-) = \lim_m ({\mf R}_m
x)(t)$. Finally, suppose that $x(\sigma_\infty(t)) = \mf d$ and that
$x(\sigma_\infty(t)-)$ does not exist. Then, by definition, $({\mf
  R}_\infty x) (t) = x(\sigma_\infty(t)\ominus)$ and for
$m>x(\sigma_\infty(t)\ominus)$ $({\mf R}_m x)(t) =
x(\sigma_\infty(t)\ominus)$. This proves the assertion.  
\end{proof}

The next statement follows from Assertions \ref{ea10} and \ref{ea02}.

\begin{asser}
\label{ea03}
Fix two trajectories $x$, $y$ in $E([0,T], S_{\mf d})$. If ${\mf R}_m
x = {\mf R}_m y$ for all $m$ large enough, then $x=y$.
\end{asser}


\smallskip
For two trajectories $x$, $y\in \bb E([0,T], S_{\mf d})$, let
\begin{equation}
\label{e06}
\begin{split}
\mb d (x,y) \;=\; \sum_{m\ge 1} \frac 1{2^m} \,
d_m (x, y)  \;,  \text{ where }  d_m (x, y) \;=\; 
d_S ({\mf R}_m x, {\mf R}_m y) \;.
\end{split}
\end{equation}

Example \ref{se08} shows that $\mb d$ is not a metric in $\bb E([0,T],
S_{\mf d})$, but the next assertion states that it is a metric in
$E([0,T], S_{\mf d})$.

\begin{asser}
\label{ea04}
The map $\mb d$ is a metric in $E([0,T], S_{\mf d})$.
\end{asser}

\begin{proof}
It is clear that $\mb d$ is finite, non-negative and symmetric, and
that $\mb d$ satisfies the triangular inequality. Suppose that $\mb
d(x,y)=0$. Then, ${\mf R}_m x = {\mf R}_m y$ for all $m\ge 1$, and,
by Assertion \ref{ea03}, $x=y$. 
\end{proof}

\begin{example}
\label{se07}
Fix $t_0<T$ and let $x_n \in D([0,T], S_{\mf d})$ be the sequence
given by
\begin{equation*}
x_n \;=\; \mb 1\{[0,t_0)\} \;+\; n \, \mb  1\{[t_0, t_0 + n^{-1})\} 
\;+\; \mb 1\{[t_0 + n^{-1}, T]\}\;.
\end{equation*}
While this sequence does not converges in the Skorohod topology, it
converges to the constant trajectory equal to $1$ in the metric $\mb
d$. In contrast and as required, for $\ell\in \bb N$, $\ell\not =1$,
the sequence
\begin{equation*}
y_n \;=\; \mb 1\{[0,t_0)\} \;+\; \ell \, \mb  1\{[t_0, t_0 + n^{-1})\} 
\;+\; \mb 1\{[t_0 + n^{-1}, T]\}
\end{equation*}
does not converge in the metric $\mb d$. 



The undesirable aspect of the metric $\mb d$ is that the sequence 
\begin{equation*}
z_n \;=\; \mb 1\{[0,t_0)\} \;+\; n \, \mb 1\{[t_0, T]\} 
\end{equation*}
also converges to the constant trajectory equal to $1$. To exclude
such cases, we shall introduce in the next section a subset of
trajectories in $E([0,T], S_{\mf d})$ which spend only a negligible
amount of time in $\mf d$ and we shall introduce compactness
conditions which ensure that the limit points of a sequence of
trajectories belongs to this set. These compactness conditions will
exclude sequences as $z_n$ which spend uniformly a non-negligible
amount of time in a set $S^c_m$ for some $m$.
\end{example}

We conclude this section proving in Proposition \ref{se06} below that
the path space $E([0,T], S_{\mf d})$ endowed with the metric $\mb d$
is complete and separable. Recall that we denote by $\Lambda$ the set
of increasing and continuous functions $\lambda:[0,T]\to [0,T]$ such
that $\lambda(0)=0$, $\lambda (T)=T$.

\begin{asser}
\label{ea05}
Let $x$ be a trajectory in $D([0,T],S_{m+1})$ and fix
$\lambda\in\Lambda$. Then, ${\mf R_m} (x\circ \lambda) = ({\mf R_m}
x)\circ \lambda$. The same identity holds for a trajectory $x$ in
$D([0,T],S_{\mf d})$.
\end{asser}

\begin{proof}
Since $x \in D([0,T], S_{m+1})$, there exist $k\ge 1$,
$0=t_0<t_1<\dots <t_k=T$, and $\ell_0, \dots , \ell_k\in S_{m+1}$ such
that $\ell_i \not = \ell_{i+1}$, $0\le i \le k-2$, and
\begin{equation}
\label{e11}
x (t) \;=\; \sum_{i=0}^{k-1} \ell_i \, \mb 1\{ [t_i, t_{i+1}) \} (t)
  \; +\; \ell_k \mb 1\{ t=t_k\} \;.
\end{equation}
Note that $\ell_{k-1}$ may be equal to $\ell_k$ in which case $x$ is
left continuous at $T$. It is easy to obtain from this formula
explicit expressions for ${\mf R_m} (x\circ \lambda)$ and for $({\mf
  R_m} x)\circ \lambda$ and to check that they are equal. 

Consider now a trajectory $x$ in $D([0,T],S_{\mf d})$. Fix
$\lambda\in\Lambda$ and $m\in S$. Recall that we denote by $\sigma^y_m
(t)$ the last visit to $S_m$ before time $t$ for the trajectory
$y$. It is easy to verify that $\sigma^{x\lambda}_m(t) =
\lambda^{-1}(\sigma^{x}_m(\lambda\,t))$, where $x\lambda = x \circ
\lambda$, $\lambda\,t = \lambda(t)$. 

Fix $t\in [0,T]$ and suppose that $x(s)\not\in S_m$ for $0\le s\le
\lambda(t)$. In this case, $x\lambda (s)\not\in S_m$ for $0\le s\le t$
and $({\mf R_m} (x \lambda)) (t) = 1 = ({\mf R_m} x) (\lambda t)$.

If $x(\lambda(t))\in S_m$, $({\mf R_m} (x \lambda)) (t) = (x \lambda)
(t) = ({\mf R_m} x) (\lambda t)$. It remains to consider the case in
which $x(s)\in S_m$ for some $0\le s\le \lambda(t)$ and
$x(\lambda(t))\not\in S_m$. By \eqref{e23}, and since
$y(\lambda^{-1}(s)-) = y\lambda^{-1}(s-)$,
\begin{equation*}
\begin{split}
& ({\mf R_m} (x \lambda)) (t) \;=\; (x \lambda)
(\sigma^{x\lambda}_m(t)-) \;=\; (x \lambda)
(\lambda^{-1}(\sigma^{x}_m(\lambda t))-)  \\
&\qquad \;=\;
x(\sigma^{x}_m(\lambda t)-) \;=\; ({\mf R_m} x) (\lambda(t))\;,
\end{split}
\end{equation*}
which proves the claim. 
\end{proof}

\begin{lemma}
\label{ea06}
The map ${\mf R}_m : D([0,T],S_{m+1}) \to D([0,T],S_m)$ is continuous
for the Skorohod topology.
\end{lemma}

\begin{proof} 
Let $x_n$ be a sequence of trajectories in $D([0,T],S_{m+1})$ which
converges in the Skorohod topology to $x$. We will prove that the
sequence of trajectories ${\mf R_m} x_n$ in $D([0,T],S_{m})$ converges
in the Skorohod topology to ${\mf R_m} x$. 

Recall the notation introduced in the beginning of this section.  Fix
$\epsilon < [m(m+1)]^{-1}$. Since $x_n$ converges to $x$, there exists
$n_0$ such that for all $n\ge n_0$
\begin{equation*}
\max \Big\{  \Vert x_n - x \lambda \Vert_\infty \,,\,
\Vert\lambda\Vert^o \Big\} \;<\; \epsilon\;,
\end{equation*}
for some $\lambda\in\Lambda$. Since we chose $\epsilon <
[m(m+1)]^{-1}$, we must have that $x_n = x \lambda$ so that ${\mf R_m}
x_n = {\mf R_m} (x \lambda)$. Since by Assertion \ref{ea05}, ${\mf
  R_m} (x \lambda) = ({\mf R_m} x)\circ \lambda$, we conclude that
\begin{equation*}
d_S( {\mf R_m} x, {\mf R_m} x_n) \;\le\; 
\max \Big\{ \Vert ({\mf R_m} x_n) - ({\mf R_m} x) \circ
\lambda  \Vert_\infty \,,\,
\Vert\lambda\Vert^o \Big\} \;<\; \epsilon\;,
\end{equation*}
which proves the lemma. 
\end{proof}

Fix $1\le k<m$, a trajectory $y\in D([0,T], S_m)$, and $t\in
(0,T]$. Let $x = \mf R_k y$. If $y(t)\le k$, then $x(t)=y(t)$, while
$x(t)\not = y(t)$ if $y(t)>k$ because in this latter case $x(t)\le k <
y(t)$. Hence, if $(\mf R_k y)(t)\not = y(t)$, $y(t)$ is necessarily
greater than $k$.

\begin{asser}
\label{ea08}
Let $y$ be a trajectory in $D([0,T], S_m)$, $m\ge 2$, and let $x= {\mf
  R}_{m-1} y$. Suppose that $x$ is discontinuous at $t\in
(0,T]$. Then, $y(t)=x(t)$ and $y$ is discontinuous at $t$.
\end{asser}

\begin{proof}
We first show that $y(t)=x(t)$ if $x$ is discontinuous at $t\in
(0,T]$. We proceed by contradiction.  Fix $t\in (0,T]$ and suppose
that $y (t) \not = x (t)$. By the remark made just before the
assertion, $y (t) = m$. We want to show that $x$ is continuous at $t$.
Since $y$ belongs to $D([0,T], S_m)$, $y$ can be represented as in
\eqref{e11}. By definition of ${\mf R}_{m-1}$, the only points where
$x$ can be discontinuous are the points $t_i$, $1\le i\le k$.  If
$t=t_i$ and $x(t_i) \not = y(t_i)$, then $y(t_i) = m$, $y(t_{i-1}) \in
S_{m-1}$ (because $y(t_{i-1}) \in S_m$ and $y(t_{i-1}) \not = y(t_i) =
m$) so that $x(t_i) = y(t_{i-1}) = x(t_{i-1}) = x(t_i-)$ and $x$ is
left-continuous at $t_i$.

We now prove the second claim of the assertion. Fix $t\in (0,T]$ and
suppose that $x$ is discontinuous at $t$. By the first part of the
claim, $y(t)=x(t)\in S_{m-1}$.  By definition of ${\mf R}_{m-1}$, $y
(t-) = x(t-)$ or $y(t-)=m$. In the first case $y$ is discontinuous at
$t$ because so is $x$. In the second case $y$ is also discontinuous at
$t$ because $y(t)\in S_{m-1}$. 
\end{proof}

\begin{figure}[htb]
  \centering
  \def\svgwidth{350pt}
  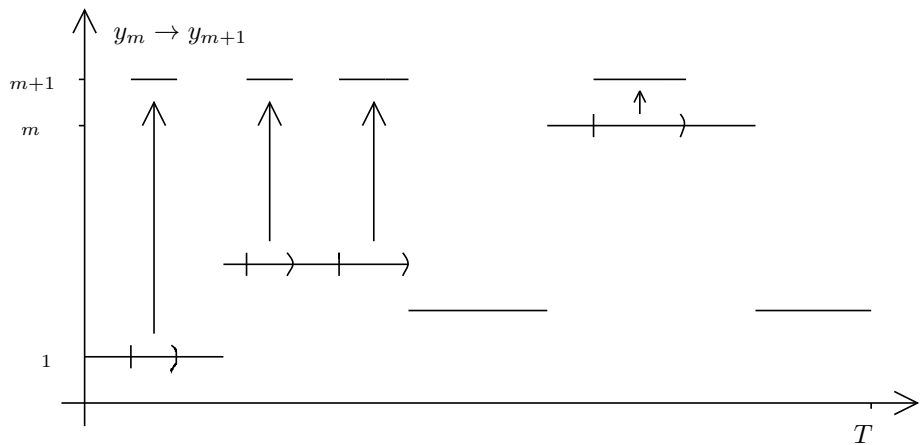
  \caption{The figure illustrates how the trajectory $y_{m+1}$ is
    obtained from the trajectory $y_m$ in Lemma \ref{se05}. The value
    of $y_{m+1}$ is set to be $m+1$ in some intervals $[s,t)$ strictly
    contained in a constancy interval of $y_m$: $y_{m+1}(r) = m+1$ for
    $s\le r <t$, where $a<s<t\le b$ and $y_m$ is constant in $[a,b)$. In
    particular, either $y_{m+1}(t) = y_m (t)$ or $y_{m+1}(t) =
    m+1$. The points $t=0$ and $t=T$ are special. For example, an
    interval $[0,r)$ can be lifted to $m+1$ if $y_m$ is constant in
    $[0,a)$, $r\le a$, and if $y_m =1$ on $[0,a)$. Note that $\mf R_m
    y_{m+1} = y_m$.}
\label{efig3}
\end{figure}

\begin{lemma}
\label{se05}
Let $y_m \in D([0,T], S_m)$, $m\ge 1$, be a sequence of trajectories
such that ${\mf R}_m y_{m+1} = y_m$ for all $m\ge 1$. Then, there
exists a trajectory $y$ in $E([0,T], S_{\mf d})$ such that ${\mf R}_m
y = y_m$ for all $m\ge 1$.
\end{lemma}

\begin{proof}
Since ${\mf R}_m x \le x$, the sequence $y_m$ is increasing and has
therefore a pointwise limit, denoted by $y$. Figure \ref{efig3}
illustrates how the trajectory $y_{m+1}$ is obtained from $y_m$. The
precise mechanism is presented below in the proof. 

Suppose that $y(t) = n \in S$ for some $t\in [0,T]$. In this case
$y_m(t)=n$ for all $m\ge n$. Indeed, if $y_{m_0}(t) \not = n$ for some
${m_0}>n$, then for all $m\ge m_0$, either $y_m(t)=y_{m_0}(t)$ or
$y_m(t)=m>n$, which contradicts the fact that $\lim_m y_m(t)=y(t)=n$.

There exists $1\le m_0\le \infty$ such that $y_m(0)=1$ for $m<m_0$ and
$y_m(0)=m_0$ for $m\ge m_0$. For any trajectory $x$, by our convention
in the definition of ${\mf R}_m$,
\begin{equation*}
({\mf R}_m x)(0) \;=\;
\begin{cases}
x(0) & \text{if $x(0)\le m$,} \\
1 & \text{if $x(0)>m$.}
\end{cases}
\end{equation*}
Therefore $y_m(0) = ({\mf R}_m y_{m+1})(0)$ satisfies the relation
\begin{equation}
\label{e10}
y_m (0) \;=\;
\begin{cases}
y_{m+1}(0) & \text{if $y_{m+1}(0)\le m$,} \\
1 & \text{if $y_{m+1}(0) = m+1$.}
\end{cases}
\end{equation}
Let $m_0 = \min\{j\ge 1 : y_j(0)\not = 1\}$. Assume that $m_0<\infty$,
otherwise there is nothing to be proven. By \eqref{e10} for $m=m_0-1$,
$y_{m_0}(0)=m_0$, and by definition of $m_0$, $y_k(0)=1$ for
$k<m_0$. By \eqref{e10} for $m=m_0$, $y_{m_0+1}(0) = y_{m_0}(0) =
m_0$. Repeating this argument, we conclude that $y_{k}(0)= m_0$ for
all $k\ge m_0$, as claimed.

The trajectory $y$ has a soft left-limit at each point $t\in (0,T]$.
Fix $t\in (0,T]$ and suppose that there exists an increasing sequence
$t_j$ converging to $t$ and such that $y(t_j)\to n\in S$. For $j$
large enough $y(t_j) = n$. We assume, without loss of generality, that
this holds for all $j$: $y(t_j)=n$ for all $j\ge 1$. By the penultimate
paragraph, $y_m(t_j)=n$ for all $m\ge n$ and $j\ge 1$. This proves
that $y_m(t-)=n$ for all $m\ge n$. In particular, by Remark
\ref{se03}, $y$ has a soft left-limit at $t$.

It is not difficult to construct an example of a sequence $y_m$ for
which $y$ has a soft left-limit at $t\in (0,T]$, but not a left-limit,
i.e., a sequence $y_m$ for which there exist increasing sequences
$t_j$, $t'_j$ converging to $t$ and such that $y(t_j)\to n\in S$,
$y(t'_j)\to \mf d$.

The trajectory $y$ is soft right-continuous. Fix $t\in [0,T)$ and
suppose that there exists a decreasing sequence $t_j$ converging to
$t$ and such that $y(t_j)\to n\in S$. The argument presented above
shows that $y_m(t)=n$ for all $m\ge n$, which proves, in view of
Remark \ref{se03}, that $y$ has a soft right-limit at $t$ equal to
$n$. Since $y_m(t)=n$ for all $m\ge n$, $y(t)=n$, which proves that
$y$ is soft right-continuous at $t$.

Fix $t\in (0,T]$ and assume that there exists $m$ for which $y_m$ is
discontinuous at $t$. By Assertion \ref{ea08}, $y_{m+1} (t) = y_{m}
(t)$ and $y_{m+1}$ is discontinuous at $t$. Repeating this argument,
we conclude that $y_n(t) = y_m(t)$ for all $n\ge m$ so that $y(t) =
y_m(t)\in S$.

The trajectory $y$ belongs to $E([0,T], S_{\mf d})$. We proved above
that $y(0)\in S$. Assume that $y(t)=\mf d$ for some $t\in (0,T]$. By
the previous paragraph, $t$ is a continuity point of $y_m$ for every
$m$. Denote by $[\ell_m,r_m)$ the largest interval which contains $t$
and in which $y_m$ is constant. $\ell_m$ is a non-decreasing sequence
bounded above by $t$. Denote by $\ell$ its limit. It is clear that
$\ell = \sigma^y_\infty(t)$, that $y(\ell)=\mf d$ and that
$y(\ell-)=\mf d$. We claim that $\ell>0$. By construction, there
exists $m_0$ such that $y_m(0)=1$ for $m<m_0$ and $y_m(0)=m_0$ for
$m\ge m_0$. As $y(t)=\mf d$, there exists $m_1$ such that $y_m(t)
>m_0$ for $m\ge m_1$. In particular, for $m\ge m_1$, $\ell_m >0$,
which proves that $\ell>0$ and that $y$ belongs to $E([0,T], S_{\mf
  d})$.

It remains to show that ${\mf R_m} y = y_m$ for all $m\ge 1$. Fix
$m\ge 1$ and $t\in [0,T]$. If $t$ is a point of discontinuity of
$y_m$, by Assertion \ref{ea08}, $y_n(t) = y_m(t)$ for all $n\ge m$ so
that $y(t)=y_m(t) \in S_m$ and $({\mf R_m} y)(t) = y_m (t)$. If $t$ is
a continuity point of $y_m$, as above, let $[\ell_m, r_m)$ the largest
constancy interval of $y_m$ which contains $t$. If $\ell_m>0$,
$\ell_m$ is a discontinuity point of $y_m$ so that $y(\ell_m) =
y_m(\ell_m)\in S_m$. By definition of the sequence $y_k$, for $k>m$
and $\ell_m \le s\le t$, $y_k(s)=y_m(s) = y_m(\ell_m)$ or
$y_k(s)>m$. Hence, for $\ell_m \le s\le t$, $y(s)= y_m(\ell_m)$
or $y(s)>m$, so that $({\mf R_m} y)(t) = y_m(\ell_m) = y_m (t)$. If
$\ell_m = 0$ and $y_m(0)\not =1$, the same argument holds since the
sequence $y_k(0)$, $k\ge m$, is constant by the assertion above
\eqref{e10}. If $\ell_m = 0$ and $y_m(0) =1$, the argument can be
adapted even if the sequence $y_k(0)$ may not be constant.  By the
assertion above \eqref{e10}, for $k>m$ and $0 \le s\le t$,
$y_k(s)=y_m(s) = 1$ or $y_k(s)>m$. Hence, for $0 \le s\le t$, $y(s)=1$
or $y(s)>m$. If $y(s)>m$ for all $0\le s\le t$, by our convention in
the definition of ${\mf R_m}$, $({\mf R_m} y)(t) = 1 = y_m (t)$. If
there exists $0\le s\le t$ such that $y(s)=1$ we also have that $({\mf
  R_m} y)(t) = 1 = y_m (t)$. This concludes the proof of the
lemma. 
\end{proof}

\begin{remark}
\label{ea16}
Let $y\in E([0,T], S_{\mf d})$, $y_m(t)$, $m\ge 1$, be the
trajectories appearing in the statement of the previous lemma.  It
follows from Assertion \ref{ea08} that if $y_m(t)$ is discontinuous at
$t\in (0,T]$, the sequence $\{y_\ell (t) : \ell\ge m\}$ is constant
and $y(t) = y_m(t)$. 

Fix $m\ge 1$. For $\ell\ge m$, since $y_\ell$ belongs to $D([0,T],
S_{\ell})$, the set $I_\ell = \{t\in [0, T] : y_\ell(t) = m\}$ is the
union of intervals $[s^\ell_k,t^\ell_k)$, $1\le k \le n_\ell$. The
last interval may be closed, all the other ones are closed on the left
and open on the right.  The intervals are disjoint,
$t^\ell_k<s^\ell_{k+1}$, $1\le k<n_\ell$, and nondegenerate,
$s^\ell_k<t^\ell_k$, excepet the last one which can be reduced to a
point.

The sequence $I_\ell$ is decreasing, $I_{\ell+1}\subset I_\ell$, and a
left end-point $s^\ell_k$ of $I_\ell$ belongs to $\cap_{\ell'\ge
  \ell}I_{\ell'}$: $s^\ell_k = s^{\ell+1}_j$ for some $1\le j \le
n_{\ell+1}$. In particular, the number of intervals may only increase,
$n_\ell \le n_{\ell+1}$. The set $\{t\in [0,T]: y(t) =m\}$ is equal to
the limit of the sets $I_\ell$,
\begin{equation*}
\{t\in [0,T]: y(t) =m\} \;=\; \bigcap_{\ell \ge m} \{t\in [0,T]:
  y_\ell(t) =m\}\; .
\end{equation*}
\end{remark}

\begin{proposition}
\label{se06}
The space $E([0,T], S_{\mf d})$ endowed with the metric $\mb d
(x,y)$ is complete and separable.
\end{proposition}

\begin{proof}
Consider a Cauchy sequence $\{x_n : n\ge 1\}$ in $E([0,T], S_{\mf d})$
for the metric $\mb d$. By definition of $\mb d$, for each $m\ge 1$,
${\mf R_m} x_n$ is a Cauchy sequence in $D([0,T], S_m)$ for the metric
$d_S$. Since this space is complete, there exists $y_m \in D([0,T],
S_m)$ such that ${\mf R_m} x_n\to y_m$ as $n\uparrow\infty$. By
Lemma \ref{ea06}, ${\mf R_m} y_{m+1} = y_m$.  Hence, by Lemma
\ref{se05}, there exists $y\in E([0,T], S_{\mf d})$ such that ${\mf
  R_m} y = y_m$ for all $m\ge 1$. Therefore, ${\mf R_m} x_n\to y_m =
{\mf R_m} y$, which implies that $x_n$ converges to $y$ in $E([0,T],
S_{\mf d})$. This proves completeness.

The separability of $E([0,T], S_{\mf d})$ follows from the
separability of each set $D([0,T]$, $S_m)$. Since the set $D([0,T],
S_m)$, $m\ge 1$, endowed with the metric $d_S$ is separable, for each
$m\ge 1$ there exists a sequence of trajectories $x_{m,n}$, $n\ge 1$,
which is dense in $D([0,T], S_m)$ for the metric $d_S$. We claim that
the countable set of trajectories $x_{m,n}$, $n\ge 1$, $m\ge 1$ is
dense.

Fix a trajectory $x$ in $E([0,T], S_{\mf d})$ and $\epsilon >0$. Take
$m\ge 1$ such that $2^{-m}<\epsilon$ and $x_{m,n}$ in $D([0,T], S_m)$
such that $d_S(x_{m,n}, {\mf R_m} x) < \min\{\epsilon,
[m(m-1)]^{-1}\}$. There exists $\lambda$ in $\Lambda$ such that
\begin{equation*}
\max\{ \Vert x_{m,n} - ({\mf R_m} x)\circ \lambda \Vert_\infty\,,\, 
\Vert \lambda\Vert^o \} \; < \; \min\{\epsilon, [m(m-1)]^{-1}\}\;.
\end{equation*}
Since $\Vert x_{m,n} - ({\mf R_m} x)\circ\lambda \Vert_\infty <
[m(m-1)]^{-1}$, $x_{m,n} = ({\mf R_m} x) \circ \lambda$. Hence, by
Assertion \ref{ea05}, for $\ell \le m$, ${\mf R_\ell} x_{m,n} = {\mf
  R_\ell} [({\mf R_m} x) \circ \lambda] = ({\mf R_\ell} x) \circ
\lambda$. In particular,
\begin{equation*}
d_S({\mf R_\ell} x_{m,n}, {\mf R_\ell} x) \;\le\; \Vert \lambda\Vert^o
\;<\; \epsilon\;.
\end{equation*}
Putting together the previous estimates, as $d_S(x,y) \le 1$ for any
pair of trajectories in $D([0,T], S_\ell)$, we obtain that
\begin{equation*}
\sum_{\ell\ge 1} \frac 1{2^\ell}\, d_S({\mf R_\ell} x_{m,n}, {\mf R_\ell} x)
\;\le\; \sum_{\ell=1}^m \frac 1{2^\ell}\, d_S({\mf R_\ell} x_{m,n}, {\mf R_\ell} x)
\;+\; \epsilon \;\le\; 2\epsilon\;.
\end{equation*}
This concludes the proof of the proposition.
\end{proof}


By extension, we call soft topology the topology in $E([0,T], S_{\mf
  d})$ induced by the metric $\mb d$.  A sequence of trajectories
$x_n\in E([0,T], S_{\mf d})$ which converges converges in the soft
topology is said to s-converges. Denote by $\mc B$ the Borel
$\sigma$-algebra of subsets of $E([0,T], S_{\mf d})$ spanned by the
open sets of the soft topology.


\begin{asser}
\label{ea15}
The subspaces $D([0,T], S_m)$, $m\ge 1$, of $E([0,T], S_m)$ are closed
for the soft topology.
\end{asser}

\begin{proof}
Consider a sequence $x_n\in D([0,T], S_m)$ s-converging to $x$. For
all $\ell\ge 1$, ${\mf R_\ell} x_n$ converges in the Skorohod topology
to ${\mf R_\ell} x$. Since $x_n$ belongs to $D([0,T], S_m)$, ${\mf
  R_\ell} x_n = {\mf R_m} x_n$ for $\ell \ge m$, so that $x =
\lim_\ell {\mf R_\ell} x = \lim_\ell \lim_n {\mf R_\ell} x_n = \lim_n
{\mf R_m} x_n = {\mf R_m} x \in D([0,T], S_m)$.
\end{proof}

\section{The space $D^*([0,T], S_{\mf d})$}
\label{esec2}


Denote by $D^*([0,T], S_{\mf d})$ the subset of all trajectories in
$D([0,T], S_{\mf d})$ which spend no time at $\mf d$ and which are
continuous at time $T$:
\begin{equation*}
D^*([0,T], S_{\mf d}) \;=\; \Big\{ x\in D([0,T], S_{\mf d}) :
\Lambda_T(x) = 0 \,,\, x(T-) = x(T) \Big\}\;,
\end{equation*}
where
\begin{equation*}
\Lambda_T(x) \;=\; \int_0^T \mb 1\{x(s) = \mf d\} \, ds\;.
\end{equation*}


Since a trajectory $x$ in $D^*([0,T], S_{\mf d})$ spends no time at
$\mf d$, $\sigma_\infty^x(t) = t$ for all $t\in [0,T]$. In particular, by
definition of the map $\mf R_\infty$, for $x$ in $D^*([0,T], S_{\mf d})$
\begin{equation}
\label{e16}
(\mf R_\infty x)(t)\;=\;
\begin{cases}
x(t) & \text{if $x(t)\in S$,} \\
x(t-) & \text{if $x(t) = \mf d$.}
\end{cases}
\end{equation}
Therefore, $(\mf R_\infty x)(t) \not = x(t)$ only if $x(t) = \mf d
\not = x(t-)$ and $(\mf R_\infty x)(T) = x(T)$.

\begin{asser}
\label{ea11}
The map $\mf R_\infty : D^*([0,T], S_{\mf d}) \to E([0,T], S_{\mf d})$
is one-to-one.
\end{asser}

\begin{proof}
Fix two trajectories $x$, $y\in D^*([0,T], S_{\mf d})$ and suppose
that $\mf R_\infty x = \mf R_\infty y$. Let $A=\{ t\in [0,T] : x(t)=
\mf d \text{ or } y(t) = \mf d\}$. By \eqref{e16}, $x(t)=y(t)$ for
$t\not \in A$. Hence, since the set $A$ has measure zero and since $x$
and $y$ are right continuous, $x(t) = y(t)$ for $t\in [0,T)$. On the
other hand, as we have seen just below \eqref{e16}, $x(T) = (\mf R_\infty
x)(T) = (\mf R_\infty y)(T) = y(T)$. 
\end{proof}


\begin{definition}
  Denote by $E^*([0,T], S_{\mf d})$ the range of the map $\mf R_\infty
  : D^*([0,T], S_{\mf d}) \to E([0,T]$, $S_{\mf d})$.
\end{definition}

Fix $x\in E([0,T], S_{\mf d})$. By Assertions \ref{ea10} and
\ref{ea02}, $x= \lim_m {\mf R}_m x$. In particular $x:[0,T]\to \bb R$
is Borel measurable and $\Lambda_T(x)$ is well defined.

\begin{lemma}
\label{ea13}
A trajectory $y$ in $E([0,T], S_{\mf d})$ belongs to $E^*([0,T],
S_{\mf d})$ if and only if
\begin{itemize}
\item[{\rm (a)}] $y$ has left and right-limits at every point;
\item[{\rm (b)}] If $y(t+) = \mf d$ for some $t\in [0,T)$, then
  $y(t)=y(t-)$;
\item[{\rm (c)}] $y$ is continuous at $T$;
\item[{\rm (d)}] $\Lambda_T(y)=0$.
\end{itemize}

\end{lemma}

\begin{proof}
Fix a trajectory $y$ in $E^*([0,T], S_{\mf d})$. Let $x\in D^*([0,T],
S_{\mf d})$ such that $y = {\mf R_\infty} x$. It follows from
\eqref{e16} that $y(t+)=x(t+)$, $y(t-)=x(t-)$, which proves
(a). Assume that $y(t+) = \mf d$ for some $t\in [0,T)$. As we just
have seen, $x(t+)= \mf d$. Since $x$ is right continuous, $x(t)=\mf
d$. Thus, by \eqref{e16}, $y(t)=x(t-)$. By the first part of the
proof, $x(t-)=y(t-)$, so that $y(t)=y(t-)$, which proves (b). To
verify (c), recall from \eqref{e16} that $y(T) = x(T)$ and from the
first part of the proof that $y(T-) = x(T-)$. Since $x$ belongs to
$D^*([0,T], S_{\mf d})$, $x(T)=x(T-)$ so that $y(T)=y(T-)$. Finally, since
$y(t)\in S$ whenever $x(t)\in S$, $x(t)=\mf d$ if $y(t)=\mf d$, and
$\Lambda_T(y) \le \Lambda_T(x) =0$.

Conversely, let $y$ be a trajectory in $E([0,T], S_{\mf d})$ which
fulfills conditions (a)--(d). Let $x$ be the trajectory defined by
$x(t) = y(t+)$, $0\le t<T$, $x(T)=y(T)$.  We claim that $x\in
D^*([0,T], S_{\mf d})$. By definition, $x$ is right continuous and has
left limits, and $x(t+) = y(t+)$, $x(t-) = y(t-)$.  Therefore, $x\in
D([0,T], S_{\mf d})$, and, by assumption (c), $x(T)=x(T-)$.

By definition of $x$,
\begin{equation*}
\Lambda_T(x) \;=\; \int_0^T \mb 1\{y(s+) = \mf d\} \, ds\;.
\end{equation*}
Fix $t\in [0,T)$ such that $y(t+) = \mf d$. Then, either $y(t)=\mf d$
or, by assumption (b), $y(t-) = y(t) \in S$. The first set of points
has Lebesgue measure zero because $\Lambda_T(y)=0$ by assumption
(d). The second set is at most countable because $y$ is constant on an
interval $[t-\epsilon, t]$ if $y(t-) = y(t) \in S$. This proves that
$\Lambda_T(x) =0$.

It remains to show that $\mf R_\infty x = y$. Suppose that $x(t)\in
S$. By the definition \eqref{e16} of $\mf R_\infty$, $(\mf R_\infty
x)(t) = x(t) = y(t+)$. Since $y$ is soft right-continuous and since
$y$ has a right-limit which belongs to $S$, $y(t+) = y(t)$, so that
$(\mf R_\infty x)(t) = y(t)$.  Suppose now that $x(t)=\mf d$, so that
$y(t+)=\mf d$. By definition \eqref{e16} of $\mf R_\infty$, $(\mf
R_\infty x)(t) = x(t-) = y(t-)$. Since $y(t+)=\mf d$, by assumption
(b), $y(t-)=y(t)$ so that $(\mf R_\infty x)(t) = y(t)$. 
\end{proof}

The set $E^*([0,T], S_{\mf d})$ is clearly not closed for the soft
topology. Lemma \ref{se16} shows that it belongs to $\mc B$, and Lemma
\ref{se11} provides sufficient conditions for the limit $x$ of a
converging sequence $x_n$ in $E^*([0,T], S_{\mf d})$ to belong to
$E^*([0,T], S_{\mf d})$.

For a trajectory $x\in D([0,T], S_{m})$ and $1\le j\le m$, denote by
$\mf N_j = \mf N_j(x)$ the number of visits of $x$ to $j$ in the time
interval $[0,T]$, and denote by $T_{j,1}, \dots, T_{j,\mf N_j}$ the
holding times at $j$. Hence, if the trajectory $x$ is given by
\begin{equation*}
x (t) \;=\; \sum_{i=0}^{k-1} \ell_i \, \mb 1\{ [t_i, t_{i+1}) \} (t)
  \; +\; \ell_k \mb 1\{[t_k, T]\} \;,
\end{equation*}
where $0=t_0<t_1<\dots <t_k\le t_{k+1} = T$, and $\ell_i \not =
\ell_{i+1}$, $0\le i \le k-1$, and if we denote by $I_j$ the set $\{ i
\in \{0, \dots, k\} : \ell_i = j\big\}$, we have that $\mf N_j (x) =
|I_j|$. Moreover, if $\mf N_j \ge 1$ and if $I_j = \{i_1, \dots,
i_{\mf N_j}\}$, where $i_a<i_{a+1}$ for $1\le a < \mf N_j$,
\begin{equation}
\label{e18}
T_{j,1} \;=\; t(i_1+1) - t (i_1)\;,\; \dots \;,\;
T_{j,\mf N_j} \;=\; t (i_{\mf N_j}+1) - t(i_{\mf N_j})\;.
\end{equation}
In this formula, to avoid small indices we represented $t_{i_a}$ by
$t(i_a)$. By convention, $T_{j,\ell} =0$ for $\ell>\mf N_j$.

\begin{asser}
\label{ea14}
Let $x$ be a trajectory in $E^*([0,T], S_{\mf d})$. Then, for all
$\ell \ge 1$, $\mf R_\ell x$ is continuous at $T$.
\end{asser}

\begin{proof}
Fix $\ell\ge 1$.  By Lemma \ref{ea13}, $x$ is continuous at
$T$. Suppose that $x(T)\in S$. In this case, $x$ is constant in an
interval $(T-\delta, T]$, $\delta>0$, and so is $\mf R_\ell x$.

Suppose that $x(T)=\mf d$. Fix $m>\ell$. There exists $\delta>0$ such
that $x(t)\ge m$ on $(T-\delta, T]$. Hence, $\sigma_\ell$ is constant
in this interval and so is $\mf R_\ell (x)$.  
\end{proof}


\begin{asser}
\label{ea12}
The functionals $\mf N_k$, $1\le k\le m$, are continuous with respect
to the Skorohod topology in $D([0,T], S_{m})$, and the sets $\{x :
T_{j,\ell} \ge a\}$, $a>0$, are closed.
\end{asser}

\begin{proof}
Fix $1\le k\le m$, and let $\{x_n : n\ge 1\}$ be a sequence in
$D([0,T], S_{m})$ which converges to a trajectory $x$ in the
Skorohod topology. Fix $\epsilon < [m(m-1)]^{-1}$. Since $x_n$
converges to $x$, there exists $n$ sufficiently large and
$\lambda_n\in \Lambda$ such that
\begin{equation*}
\Vert x_n - x\lambda_n \Vert_\infty\;<\; \epsilon\;.
\end{equation*}
Since $\epsilon < [m(m-1)]^{-1}$ we have that $x_n = x\lambda_n$ so that
$\mf N_k (x\lambda_n) = \mf N_k (x_n)$. Since $\mf N_k (x\lambda_n) = \mf
N_k (x)$, we conclude that the sequence $\mf N_k (x_n)$ is eventually
constant and converges to $\mf N_k (x)$.

To prove that the sets $\{x : T_{j,\ell} \ge a\}$ are closed, fix
$1\le j\le m$, $\ell\ge 1$, $a>0$, and consider a sequence $x_n$
converging in the Skorohod topology to some trajectory $x$. Suppose
that $T_{j,\ell} (x_n) \ge a$ for all $n\ge 1$ and fix $0<\epsilon<
[m(m-1)]^{-1}$. There exists $\lambda_n \in \Lambda$ such that $\Vert
x_n - x\lambda_n \Vert_\infty < \epsilon$, $\Vert \lambda_n\Vert^o <
\epsilon$ for all $n$ large enough. As in the first part of the proof,
we deduce from this estimate that $x_n = x\lambda_n$ so that $\mf N_j
(x_n) = \mf N_j (x \lambda_n) = \mf N_j (x)$ and $T_{j,\ell} (x_n) =
T_{j,\ell} (x \lambda_n)$ for $n$ large enough. Since $T_{j,\ell}
(x_n) \ge a$, $\ell\le \mf N_j (x_n) = \mf N_j (x)$.  Denote by
$[s,t)$ the time interval of the $\ell$-th visit to $j$ for the
trajectory $x$, so that $T_{j,\ell} (x \lambda_n) = \lambda^{-1}_n(t) -
\lambda^{-1}_n(s)$. Since $T_{j,\ell} (x_n) = T_{j,\ell} (x
\lambda_n)$ and since $T_{j,\ell} (x_n) \ge a$, $\lambda^{-1}_n(t) -
\lambda^{-1}_n(s)\ge a$.  However, as $\Vert \lambda_n\Vert^o <
\epsilon$, $e^{-\epsilon} (t-s) \le \lambda^{-1}_n(t) -
\lambda^{-1}_n(s) \le e^{\epsilon} (t-s)$. Therefore, $T_{j,\ell} (x) =
t-s \ge e^{-\epsilon} [\lambda^{-1}_n(t) - \lambda^{-1}_n(s)] \ge
e^{-\epsilon} a$, which proves the assertion. 
\end{proof}

By expressing all conditions of Lemma \ref{ea13} in terms of the
trajectories $\mf R_\ell x$, we show in the next lemma that the set
$E^*([0,T], S_{\mf d})$ belongs to the Borel $\sigma$-algebra $\mc B$.

\begin{lemma}
\label{se16}
The set $E^*([0,T], S_{\mf d})$ belongs to the Borel $\sigma$-algebra
$\mc B$. 
\end{lemma}


\begin{proof}
To keep notation simple, denote the set $E([0,T], S_{\mf d})$ by
$E_T$.  Let 
\begin{equation*}
\begin{split}
& \Omega_1 = \{ x\in E_T : \Lambda_T(x)=0\}\;, \quad
\Omega_2 = \{ x\in E_T : x(T)=x(T-)\} \;, \\
&\quad \Omega_3 = \{x \in E_T : x \text{ has left and right-limits at
  every point }\} \;, \\
&\qquad \Omega_4 = \{x\in E_T : x(t)=x(t-) \text{ if $x(t+) =
  \mf d$ for some $t\in [0,T)$} \}\;.
\end{split}
\end{equation*}
In view of Lemma \ref{ea13}, $E^*([0,T], S_{\mf d}) = \cap_{1\le j \le
  4} \Omega_j$. To prove the lemma we have to show that the latter set
belongs to $\mc B$.

We first show that $\Omega_1$ belongs to $\mc B$. Let $F_j: E_T \to
[0,T]$, $j\ge 1$, be given by
\begin{equation*}
F_j(x) \;=\; \int_0^T \mb 1\{x(s) = j\}\, ds\;.
\end{equation*}
Since ${\mf R}_\ell x$ increases to $x$, $F_j(x) = \inf_{\ell\ge j}
F_{j,\ell}(x)$, where $F_{j,\ell}(x) = \int_{[0,T]} \mb 1\{({\mf
  R}_\ell x)(s) = j\}\, ds$. Since each function $F_{j,\ell}$ is
continuous for the soft topology, the function $F_j$ is $\mc
B$-measurable and so is $\Lambda_T = T - \sum_{j\ge 1} F_j$. This
proves that $\Omega_1$ belongs to $\mc B$.

We turn to $\Omega_2$.  We show separately that $\Omega^{2,1}_2 = \{
x\in E_T : x(T)=x(T-)=\mf d\}$ and $\Omega^{2,2}_2 =\{ x\in E_T :
x(T)=x(T-)\in S\}$ belong to $\mc B$.  By definition of $E_T$, we may
rewrite $\Omega^{2,1}_2$ as $\{ x\in E_T : x(T)=\mf d\}$. Since ${\mf
  R}_\ell x$ increases to $x$, $\Omega^{2,1}_2 = \cap_{m\ge
  1}\cup_{\ell\ge m} = \{x\in E_T : ({\mf R}_\ell x) (T) =
\ell\}$. For each $\ell\ge 1$, the set $\{x\in E_T : ({\mf R}_\ell x)
(T) = \ell\}$ is closed, which shows that $\Omega^{2,1}_2$ belongs to
$\mc B$.

On the other hand, we claim that $\Omega_2^{2,2} = \cup_{k\ge
  1}\cup_{m\ge 1}\cap_{\ell\ge m} \Omega_2^{k,m,\ell}$, where
$\Omega_2^{k,m, \ell} = \{x\in E_T : ({\mf R}_\ell x) (t) = m \,,\,
T-(1/k) \le t\le T\}$. Fix $x\in \Omega_2^{2,2}$ and set
$m=x(T)$. Since $x(T-)=m$, there exists $k\ge 1$ such that $x(t)=m$
for $T-(1/k) \le t\le T$. By definition of $\mf R_\ell$, for all $\ell
\ge m$, $({\mf R}_\ell x) (t) = m$ for $T-(1/k) \le t\le T$. Thus,
$\Omega_2^{2,2} \subset \cup_{k\ge 1}\cup_{m\ge 1}\cap_{\ell\ge m}
\Omega_2^{k,m,\ell}$. Conversely, if $x$ belongs to $\cup_{k\ge
  1}\cup_{m\ge 1}\cap_{\ell\ge m} \Omega_2^{k,m,\ell}$, there exists
$k\ge 1$ and $m\ge 1$ such that $({\mf R}_\ell x) (t) = m$ for
$T-(1/k) \le t\le T$ for all $\ell \ge m$. Since ${\mf R}_\ell x$
increases pointwisely to $x$, the same property holds for $x$, which
proves that $\Omega_2^{2,2} = \cup_{k\ge 1}\cup_{m\ge 1}\cap_{\ell\ge
  m} \Omega_2^{k,m,\ell}$. As the sets $\Omega_2^{k,m,\ell}$ are
closed, the set $\Omega_2^{2,2}$ belongs to $\mc B$.

We claim that $\Omega_3 = \cap_{m\ge 1} \cup_{k\ge 1} \cap_{\ell\ge m}
\{x \in E_T : \mf N_m ({\mf R}_\ell x) \le k\}$. Denote the right hand
side of the equality by $\Omega'_3$ and fix $x\in \Omega'_3$.  We will
show that the trajectory $x$ has left and right limits at all points
$t\in [0,T]$. Since $x$ belongs to $E([0,T], S_{\mf d})$, it is enough
to exclude the possibility that $x$ has a finite soft limit at some
point $t\in [0,T]$.  Fix $m\ge 1$ and recall Remark \ref{ea16}. As $x$
belongs to $\Omega'_3$, there exists $k \ge 1$ such that $\mf N_m(
{\mf R}_\ell x) \le k$ for all $\ell\ge m$. Since the sequence $\mf
N_m( {\mf R}_\ell x)$ increases with $\ell$, it is constant for $\ell$
large enough.  Denote by $[s^\ell_1, t^\ell_1), \dots, [s^\ell_N,
t^\ell_N)$ the $N=\mf N_m( {\mf R}_\ell x)$ time-intervals in which
${\mf R}_\ell x$ visits $m$. If $t^\ell_N=T$, the last time interval
may be closed at $T$.  By Assertion \ref{ea08}, $s^{\ell+1}_i =
s^\ell_i$, $1\le i\le N$, and $t^{\ell+1}_i \le t^\ell_i$. Since ${\mf
  R}_\ell x$ converges pointwisely to ${\mf R}_\infty x = x$, the set
$\{s\in [0,T] : x(s) = m\}$ is the union of $N$ disjoint intervals,
some of which can be reduced to a point. In particular, $m$ can not be
the finite soft limit of $x$ at some point $t\in [0,T]$. Since this
holds for every $m$, $x$ does not have a left or a right finite soft
limit at some $t\in [0,T]$.

Conversely, fix a trajectory $x$ which does not belong to
$\Omega'_3$. In this case, there exists $m\ge 1$ such that $\mf N_m
({\mf R}_\ell x)$ increases to $\infty$. By Assertion \ref{ea08}, the
set $A_m = \{t \in [0,T] : x(t)=m \not = x(t-)\}$ is countably
infinite because it contains all the left end-points of the time
intervals $[s,t)$ in which ${\mf R}_\ell x$ is constant equal to
$m$. Let $t$ be an accumulation point of $A_m$ and assume, without
loss of generality, that there exists $t_j\uparrow t$.  Then, $x(t_j)
= m$ and there exist $s_j\uparrow\infty$ such that $x(s_j) \not = m$
for all $j$. This proves that $x$ has not a left limit at $t$, proving
that $\Omega_3 = \Omega'_3$. Since the sets $\{x \in E_T : \mf N_m
({\mf R}_\ell x) \le k\}$ are closed, $\Omega_3$ belongs to $\mc B$.
 
Finally, consider the set $\Omega_4$. By definition of $E_T$,
$x(t-)=x(t)=\mf d$ if $x(t)=\mf d$. The set $\Omega_4$ may therefore
be rewritten as $\{x\in E_T : x(t-)=x(t) \text{ if } x(t+) = \mf d,
x(t)\in S \text{ for some } t\in [0,T)\}$.  Let $\Omega'_4 =
\cap_{m\ge 1} \cup_{k\ge 1} \cap_{\ell\ge m}\{x\in E_T : T_{m,j}({\mf
  R}_\ell x) \ge (1/k) \text{ for all } 1\le j\le \mf N_m ({\mf
  R}_\ell x)\}$, and $\Omega_{13}= \Omega_1 \cap \Omega_2 \cap
\Omega_3$. We claim that $\Omega_{13} \cap \Omega_4 = \Omega_{13} \cap
\Omega'_4$.

Consider a trajectory $x\in \Omega_{13} \cap \Omega'_4$, $m\ge 1$, and
recall the notation introduced when we proved that $\Omega_3$ belongs
to $\mc B$. For $\ell\ge m$, $[s^\ell_1, t^\ell_1), \dots, [s^\ell_N,
t^\ell_N)$ represent the $N=\mf N_m( {\mf R}_\ell x)$ time-intervals
in which ${\mf R}_\ell x$ visits $m$. Since $T_{m,j}({\mf R}_\ell x)
\ge (1/k)$ for $1\le j\le \mf N_m ({\mf R}_\ell x)$, $x$ is equal to
$m$ in $N$ intervals $[s_j, t_j)$ of length at least $1/k$. Since
$s^\ell_i = s^{\ell+1}_i$, taking the limit $\ell\uparrow\infty$ we
obtain that there is no $t\in [0,T)$ such that $x(t+)=\mf d$,
$x(t)=m$. Since this holds for every $m$, $x$ belongs to $\Omega_{13}
\cap \Omega_4$.

Reciprocally, consider a trajectory $x$ in $\Omega_{13} \cap
[\Omega'_4]^c$. By definition, there exist $m\ge 1$, a sequence
$\ell_k$ and an interval $[s^{\ell_k}_i, t^{\ell_k}_i)$, $1\le i \le
\mf N_m( {\mf R}_{\ell_k} x)$, such that $t^{\ell_k}_i - s^{\ell_k}_i
\le 1/k$. The number of intervals, $\mf N_m( {\mf R}_{\ell_k} x)$, is
eventually constant because $x$ belongs to $\Omega_3$. There exists,
in particular, an index $i$ such that $\lim_{k} (t^{\ell_k}_i -
s^{\ell_k}_i)=0$. Let $t=s^{\ell_k}_i$, a sequence which is constant
in view of Assertion \ref{ea08}. Since $t^{\ell_k}_i \downarrow t$,
$x(t+)$ exists and is equal to $\mf d$, $x(t) = m$ and $x(t-)\not =
m$, which proves that $x$ belongs to $\Omega_{13} \cap
[\Omega_4]^c$. This shows that $\Omega_{13} \cap \Omega_4 =
\Omega_{13} \cap \Omega'_4$. Finally, since $\{x\in E_T : T_{m,j}({\mf
  R}_\ell x) \ge (1/k) \text{ for all } 1\le j\le \mf N_m ({\mf
  R}_\ell x)\}$ is a closed set, $\Omega_{13} \cap \Omega_4$ belongs
to $\mc B$, which concludes the proof of the lemma.
\end{proof}

Pushing further the arguments used in the proof of the previous lemma,
we obtain in the next lemma sufficient conditions, all expressed only
in terms of the trajectories ${\mf R}_\ell x_n$, for the limit $x$ of a
sequence $x_n$ in $E^*([0,T], S_{\mf d})$ to belong to $E^*([0,T],
S_{\mf d})$. 

\begin{lemma}
\label{se11}
Let $\{x_n : n\ge 1\}$ be a sequence in $E^*([0,T], S_{\mf d})$ which
converges to $x\in E([0,T], S_{\mf d})$ in the metric $\mb d$. Assume
that
\begin{itemize}

\item[{\rm (a)}] 
\begin{equation*}
\lim_{m\to\infty} \sup_{\ell\ge 1} \, \sup_{n\ge 1} 
\int_0^T \mb 1\{{\mf R}_\ell x_n(s) \ge m\} \, ds \;=\; 0\;;
\end{equation*}

\item[{\rm (b)}] For each $m\ge 1$, there exists $k_m\in \bb N$ such that
$\mf N_m ({\mf R}_\ell x_n) \le k_m$ for all $\ell\ge m$ and $n\ge 1$;

\item[{\rm (c)}] For each $m\ge 1$, there exists $\epsilon_m>0$ such
  that $T_{m,k}({\mf R}_\ell x_n) \ge \epsilon_m$ for all $1\le k\le
  \mf N_m ({\mf R}_\ell x_n)$, $\ell\ge m$ and $n\ge 1$;

\end{itemize}
Then, $x$ belongs to $E^*([0,T], S_{\mf d})$.
\end{lemma}

\begin{proof}
We need to prove that the trajectory $x$ fulfills conditions (a)--(d)
of Lemma \ref{ea13}.  We first claim that $\Lambda_T(x)=0$. Fix
$\epsilon >0$. By assumption (a), there exists $m\ge 1$ such that
\begin{equation*}
\int_0^T \mb 1\{ ({\mf R}_\ell x_n) (s) \ge m\}\, ds \;\le\; \epsilon
\end{equation*}
for all $n\ge 1$, $\ell\ge 1$. Fix $\ell\ge m$. The sequence ${\mf
  R}_\ell x_n$ converges almost everywhere to ${\mf R}_\ell x$ because
it converges in the Skorohod topology. Hence, since $\mf R_\ell x_n$
takes values in a discrete set, by Fatou's lemma,
\begin{equation*}
\int_0^T \mb 1\{ ({\mf R}_\ell x) (s) \ge m\}\, ds 
\; \le\; \liminf_{n\to\infty} \int_0^T 
\mb 1\{ ({\mf R}_\ell x_n) (s) \ge m\}\, ds \;\le\; \epsilon\;.
\end{equation*}
Since ${\mf R}_\ell x$ converges pointwisely to $x$, by the dominated
convergence theorem,
\begin{equation*}
\int_0^T \mb 1\{ x (s) \ge m\}\, ds \; \le\; \epsilon\;,
\end{equation*}
so that $\Lambda_T(x) \le \epsilon$.

We now show that $x$ has left and right limits.  Since $x$ belongs to
$E([0,T], S_{\mf d})$ to prove this claim it is enough to exclude the
possibility that $x$ has a finite soft limit at some point $t\in
[0,T]$.  Fix $m\ge 1$. By assumptions (b) and (c) of this lemma, there
exist $k_m\ge 1$ and $\epsilon_m>0$ such that $\mf N_m( {\mf R}_\ell
x_n) \le k_m$ and $T_{m,k}({\mf R}_\ell x_n) \ge \epsilon_m$ for all
$1\le k\le \mf N_m ({\mf R}_\ell x_n)$, $\ell\ge m$, $n\ge 1$. Since
${\mf R}_\ell x_n$ converges in the Skorohod topology to ${\mf R}_\ell
x$, by Assertion \ref{ea12}, $\mf N_m({\mf R}_\ell x) \le k_m$ and
$T_{m,k}({\mf R}_\ell x) \ge \epsilon_m$ for all $1\le k\le \mf N_m
({\mf R}_\ell x)$, $\ell\ge m$. As the sequence $\mf N_m( {\mf R}_\ell
x)$ increases with $\ell$, it is constant for $\ell$ large enough.
Denote by $[s^\ell_1, t^\ell_1), \dots, [s^\ell_N, t^\ell_N)$ the
$N=\mf N_m( {\mf R}_\ell x)$ time-intervals in which ${\mf R}_\ell x$
visits $m$. Since $T_{m,k}({\mf R}_\ell x) \ge \epsilon_m$ for all
$k$, $t^\ell_i \ge s^\ell_i + \epsilon_m$. By Assertion \ref{ea08},
$s^{\ell+1}_i = s^\ell_i$, $1\le i\le N$, and $t^{\ell+1}_i \le
t^\ell_i$. Since ${\mf R}_\ell x$ converges pointwisely to ${\mf
  R}_\infty x = x$, the set $\{s\in [0,T] : x(s) = m\}$ is the union
of $N$ disjoint intervals of length greater or equal to $\epsilon_m$,
which are closed at the left boundary and open or closed at the right
boundary.  In particular, $m$ can not be the finite soft limit of $x$
at some point $t\in [0,T]$. Since this holds for every $m$, $x$ does
not have a left or a right finite soft limit at some time $t\in
[0,T]$. This proves condition (a) of Lemma \ref{ea13}.

We turn to condition (b) of Lemma \ref{ea13}. Suppose that
$x(t+)=\mf d$ for some $t\in [0,T)$. If $x(t)=\mf d$, since $x\in
E([0,T], S_{\mf d})$ and $\Lambda_T(x)=0$, $\sigma_\infty(t)=t$ and,
by definition of the set $E([0,T], S_{\mf d})$, $x(t-)=x(t)$. If
$x(t)=m\in S$, since $x(t+)=\mf d$, $t$ is the right endpoint of an
interval $[s_i,t_i]$ obtained as the limit of the intervals
$[s^\ell_i,t^\ell_i)$ introduced in the previous paragraph. Since the
interval is not degenerate, $x(t-)=m=x(t)$, which proves condition (b)
of Lemma \ref{ea13}.

We finally prove condition (c) of Lemma \ref{ea13}. Suppose that
$x(T)=k\in S$. In this case, since the set $\{s\in [0,T] : x(s) = k\}$
is the union of a finite number of disjoint intervals of positive
length, $x$ is continuous at $T$. Suppose now that $x(T)=\mf d$.  By
Assertion \ref{ea14}, $({\mf R}_\ell x_n)(T) = ({\mf R}_\ell x_n)(T-)$
for all $\ell\ge 1$, $n\ge 1$. Since ${\mf R}_\ell x_n$ converges to
${\mf R}_\ell x$ in the Skorohod topology, the continuity at $T$ is
inherited by ${\mf R}_\ell x$.  Denote by $[a_\ell, T]$ the constancy
interval of ${\mf R}_\ell x$ and fix $m\ge 1$. Since $x(T)=\mf d$ and
since $({\mf R}_\ell x)(T)$ converges to $x(T)$, there exists
$\ell_0\ge 1$ such that for all $\ell\ge\ell_0$, $({\mf R}_\ell x)(T)
\ge m$. By definition of $a_\ell$ and since $x\ge {\mf R}_\ell x$, for
all $a_\ell\le t\le T$, $x(t) \ge ({\mf R}_\ell x)(t) = ({\mf R}_\ell
x)(T) \ge m$. This proves that $x(T-)=\mf d =x(T)$.  Condition (c) of
Lemma \ref{ea13} is therefore in force, which concludes the proof of
the lemma.
\end{proof}

\begin{corollary}
\label{se09}
Let $x$ be a trajectory in $E([0,T], S_{\mf d})$ which satisfies
conditions {\rm (b)} and {\rm (c)} of the previous lemma and such that
$\Lambda_T(x)=0$, $(\mf R_\ell x)(T) = (\mf R_\ell x)(T-)$ for all
$\ell\ge 1$. Then, $x$ belongs to $E^*([0,T], S_{\mf d})$.
\end{corollary}

\begin{proof}
By the proof of Lemma \ref{se11}, $x$ satisfies conditions (a)--(c)
of Lemma \ref{ea13}. Since condition (d) of this assertion holds
by assumption, the corollary is proved. 
\end{proof}

\section{Weak Convergence of Probability Measures.}
\label{esec3}

We examine in this section the weak convergence of probability
measures on $E([0,T], S_{\mf d})$.  Fix $m\ge 1$ and consider a
sequence $x_n$ in $D([0,T], S_{m})$ converging to $x$ in the Skorohod
topology. Then, $x_n$ converges to $x$ in $E([0,T], S_{\mf
  d})$. Indeed,
\begin{equation*}
\begin{split}    
\mb d(x_n,x) \; &=\; \sum_{\ell\ge 1} \frac 1{2^\ell} \, d_S({\mf R}_\ell
x_n, {\mf R}_\ell x) \\
\;& =\; \frac 1{2^{m}} \, d_S(x_n, x) \;+\; \sum_{\ell= 1}^{m} 
\frac 1{2^\ell} \, d_S({\mf R}_\ell x_n, {\mf R}_\ell x) \;.
\end{split}
\end{equation*}
By hypothesis and by Lemma \ref{ea06}, this sum vanishes as
$n\uparrow\infty$.

Let $F: E([0,T], S_{\mf d}) \to \bb R$ be a continuous function for
the soft topology. Then, its restriction to $D([0,T], S_m)$, $m\ge 1$,
is continuous for the Skorohod topology. Indeed, consider a sequence
$x_n$ converging in $D([0,T], S_m)$ to $x$. By the previous paragraph,
$x_n$ converges to $x$ in the soft topology of $E([0,T], S_{\mf
  d})$. Since $F$ is continuous in this topology, $F(x_n)$ converges
to $F(x)$. 

\begin{theorem}
\label{se14}
A sequence of probability measures $P_n$ on $E([0,T], S_{\mf d})$
converges weakly in the soft topology to a measure $P$ if and only if
for each $m\ge 1$ the sequence of probability measures $P_n \circ {\mf
  R}^{-1}_m$ defined on $D([0,T], S_{m})$ converges weakly to $P \circ
{\mf R}^{-1}_m$ with respect to the Skorohod topology.
\end{theorem}



\begin{proof}
Suppose that the sequence $P_n$ converges weakly to $P$ and fix $m\ge
1$. Since ${\mf R}_m : E([0,T], S_{\mf d}) \to D([0,T], S_{m})$ is
continuous for the soft topology, $P_n \circ {\mf R}^{-1}_m$ converges
weakly to $P \circ {\mf R}^{-1}_m$.

Conversely, suppose that $P_n \circ {\mf R}^{-1}_m$ converges weakly
to $P \circ {\mf R}^{-1}_m$ for every $m\ge 1$. Fix a bounded, uniformly
continuous function $F: E([0,T], S_{\mf d}) \to \bb R$ and $\epsilon
>0$. Since $F$ is uniformly continuous, there exists $\delta>0$ such
that $|F(y)-F(x)|\le \epsilon$ if $\mb d(x,y)\le \delta$. Let $m\ge 1$
such that $2^{-(m-1)}< \delta$. Since $\mb d(x,{\mf R}_m x)\le
2^{-(m-1)}< \delta$, the difference $E_{P_n}[F(x)] - E_{P_n}[F({\mf
  R}_m x)]$ is absolutely bounded by $\epsilon$, uniformly in $n$. A
similar estimate holds for $P$ replacing $P_n$.

We have shown right before the statement of the theorem that $F:
D([0,T], S_{m}) \to \bb R$ is continuous for the Skorohod topology. As
$P_n \circ {\mf R}^{-1}_m$ converges weakly to $P \circ {\mf
  R}^{-1}_m$ in the Skorohod topology, and since $F$ is bounded and
continuous, there exists $n_0$ such that for all $n\ge n_0$,
$|E_{P_n}[F({\mf R}_m x)] - E_{P}[F({\mf R}_m x)]|\le \epsilon$.
Putting together the previous estimates we conclude that for all $n\ge
n_0$,
\begin{equation*}
\big| \, E_{P_n}[F(x)] - E_{P}[F(x)]\, \big| \;\le\; 3\, \epsilon\;,
\end{equation*}
which concludes the proof of the theorem. 
\end{proof}

\begin{theorem}
\label{se12}
Let $\{P_n : n\ge 1\}$ be a sequence of probability measures on the
path space $E^*([0,T], S_{\mf d})$ which converges weakly to a measure
$P$ in $E([0,T], S_{\mf d})$ endowed with the soft topology. Assume
that
\begin{itemize}

\item[{\rm (a)}] 
\begin{equation*}
\lim_{m\to\infty} \limsup_{\ell\to\infty} \limsup_{n\to\infty} E_{P_n}
\Big[ \int_0^T \mb 1\{ ({\mf R}_\ell x)(s) \ge m\} \, ds \Big]
\;=\; 0\;;
\end{equation*}

\item[{\rm (b)}] For each $m\ge 1$, 
\begin{equation*}
\lim_{k\to\infty}\limsup_{\ell\to\infty}\limsup_{n\to\infty} 
P_n \big[ \, \mf N_m ({\mf R}_\ell x) \ge k \,\big] \;=\; 0\;;
\end{equation*}

\item[{\rm (c)}] For each $m\ge 1$, 
\begin{equation*}
\lim_{\epsilon\to 0}\limsup_{\ell\to\infty}\limsup_{n\to\infty} 
P_n \Big[ \, \bigcup_{k=1}^{\mf N_m ({\mf R}_\ell x)} 
\{T_{m,k}({\mf R}_\ell x) < \epsilon \} \,\Big] \;=\; 0\;.
\end{equation*}

\end{itemize}
Then, $P$ is concentrated on $E^*([0,T], S_{\mf d})$.
\end{theorem}

\begin{proof}
It is not difficult to show that for each $m\le \ell$ the map $y \to
\int_0^T \mb 1\{ y (s) \ge m\} \, ds$ is continuous in $D([0,T],
S_{\ell})$. Therefore, the map $y \to \int_0^T \mb 1\{ ({\mf R}_\ell
y) (s) \ge m\} \, ds$ is bounded and continuous in $E([0,T], S_{\mf
  d})$. By assumption (a), given $\epsilon>0$, there exists $m_0$ such
that for all $m\ge m_0$,
\begin{equation*}
\limsup_{\ell\to\infty}  E_{P}
\Big[ \int_0^T \mb 1\{ ({\mf R}_\ell x)(s) \ge m\} \, ds \Big]
\;\le \; \epsilon \;. 
\end{equation*}
Since ${\mf R}_\ell x$ increases pointwisely to ${\mf R}_\infty x =
x$, by the monotone convergence theorem,
\begin{equation*}
E_P \big[ \Lambda_T(x) \big] \;\le\;
E_{P} \Big[ \int_0^T \mb 1\{ x (s) \ge m\} \, ds \Big]
\;\le \; \epsilon \;. 
\end{equation*}
Letting $\epsilon\downarrow 0$, we conclude that $E_P [\Lambda_T(x)] =
0$, i.e., that
\begin{equation}
\label{e19}
P [\Lambda_T(x)=0] \;=\; 1\;.
\end{equation}

By Assertion \ref{ea12}, the functionals $\mf N_m$, $m\ge 1$, are
continuous for the Skorohod topology. The sets $\{x \in D([0,T],
S_\ell) : \mf N_m(x) \ge k\} = \{x \in D([0,T], S_\ell) : \mf N_m(x)
\le k-1\}^c$ are therefore open and, by assumption (b), for every
$m\ge 1$,
\begin{equation}
\label{e14}
\lim_{k\to\infty}\limsup_{\ell\to\infty}
P \big[ \, \mf N_m ({\mf R}_\ell x) \ge k \,\big] \;=\; 0\;.
\end{equation} 
As $\mf N_m ({\mf R}_\ell x)$ is a non-decreasing sequence in $\ell$,
the set $\{\mf N_m ({\mf R}_\ell x) \ge k\}$ is contained in $\{\mf
N_m ({\mf R}_{\ell+1} x) \ge k\}$. Thus, for every $m\ge 1$,
\begin{equation*}
P\Big[ \, \bigcap_{k\ge 1} \bigcup_{\ell\ge m} 
\{\mf N_m ({\mf R}_\ell x) \ge k\} \, \Big]
\;=\; \lim_{k\to\infty} \lim_{\ell \to\infty} 
P\Big[ \, \mf N_m ({\mf R}_\ell x) \ge k \, \Big] \;=\; 0 \;,
\end{equation*}
where the last equality follows from \eqref{e14}. Since this identity
holds for every $m\ge 1$,
\begin{equation}
\label{e20}
P\Big[ \, \bigcap_{m\ge 1} \bigcup_{k\ge 1} \bigcap_{\ell\ge m} 
\{\mf N_m ({\mf R}_\ell x) \le k\} \, \Big] \;=\; 1\;.
\end{equation}

A straightforward modification of the proof of Assertion \ref{ea12}
shows that for every $\ell\ge m$, the set $\bigcap_{k=1}^{\mf N_m (y)}
\{T_{m,k}(y) \ge \epsilon \}$ is closed in $D([0,T],
S_{\ell})$. Therefore, by assumption (c),
\begin{equation*}
\lim_{\epsilon\to 0}\limsup_{\ell\to\infty}
P \Big[ \, \bigcup_{k=1}^{\mf N_m ({\mf R}_\ell x)} 
\{T_{m,k}({\mf R}_\ell x) < \epsilon \} \,\Big] \;=\; 0\;.
\end{equation*}
Since the duration of the visits to a point $m$ may only decrease as
$\ell$ increases, $\bigcup_{k=1}^{\mf N_m ({\mf R}_\ell x)}
\{T_{m,k}({\mf R}_\ell x) < \epsilon \} \subset \bigcup_{k=1}^{\mf N_m
  ({\mf R}_{\ell +1}x)} \{T_{m,k}({\mf R}_{\ell+1} x) < \epsilon
\}$. In particular, by the previous displayed equation,
\begin{equation*}
P \Big[ \, \bigcap_{j\ge 1}
\bigcup_{\ell\ge m} \bigcup_{k=1}^{\mf N_m ({\mf R}_\ell x)} 
\Big \{T_{m,k}({\mf R}_\ell x) < \frac 1j \Big\} \,\Big] \;=\; 0\;.
\end{equation*}
Since this equation holds for every $m\ge 1$, we conclude that
\begin{equation}
\label{e21}
P \Big[ \, \bigcap_{m\ge 1} \bigcup_{j\ge 1}
\bigcap_{\ell\ge m} \bigcap_{k=1}^{\mf N_m ({\mf R}_\ell x)} 
\Big \{T_{m,k}({\mf R}_\ell x) \ge \frac 1j \Big\} \,\Big] \;=\; 1\;.
\end{equation}

Since the measure $P_n$ is concentrated on $E^*([0,T], S_{\mf d})$, by
Assertion \ref{ea14}, for every $\ell$, 
\begin{equation*}
P_n \big[ \, ({\mf R}_\ell x) (T) = ({\mf R}_\ell x) (T-) \,\big] \;=\; 1\;.
\end{equation*}
As the set $\{x \in D([0,T], S_{\ell}) : x(T)=x(T-)\}$ is closed for
the Skorohod topology, by Theorem \ref{se14}, for every $\ell\ge 1$,
\begin{equation*}
P \big[ \, ({\mf R}_\ell x) (T) = ({\mf R}_\ell x) (T-) \,\big] \;=\; 1\;,
\end{equation*}
so that
\begin{equation}
\label{e22}
P \Big[ \, \bigcap_{\ell\ge 1}
\{ ({\mf R}_\ell x) (T) = ({\mf R}_\ell x) (T-) \}
\,\Big] \;=\; 1\;.
\end{equation}

Denote by $A$ the intersection of the events with full measure
appearing in \eqref{e19}, \eqref{e20}, \eqref{e21}, \eqref{e22}.  By
Corollary \ref{se09}, any trajectory in $A$ belongs to $E^*([0,T],
S_{\mf d})$. This proves the theorem. 
\end{proof}

In view of condition (b), to prove condition (c) of Theorem
\ref{se12}, it is enough to show that for each $k, m\ge 1$,
\begin{equation}
\label{e28}
\lim_{\epsilon\to 0} \limsup_{\ell \to\infty} \limsup_{n\to\infty} 
P_n \big[ \, T_{m,k}(\mf R_\ell x) < \epsilon \,\big] \;=\; 0\;.
\end{equation}

\section{Applications}
\label{esec4}

In this section, we apply Theorems \ref{se14} and \ref{se12} to prove
the convergence in the soft topology of the order parameter $\bb
X_N(t)$ to a Markov chain $\bb X(t)$ in the two models presented in
the introduction.  We examine first the case of a finite number of
wells. Recall the set-up introduced in Subsection 1.4.

\begin{theorem}
\label{se15}
Let $\nu_N$ be a sequence of probability measures on $\ms E_N$ such
that $\nu_N \circ \Psi^{-1}$ converges to a probability measure $\nu$
on $S_L$, and let $\theta_N$ be a sequence of real positive numbers.
Assume that the sequence of probability measures $\mb P_{\nu_N} \circ
(\bb X^{\rm T}_N)^{-1}$, defined on $D([0,T], S_L)$, converges in the
Skorohod topology to a measure $\bb P_\nu$ which corresponds to a
$S_L$-valued continuous-time Markov chain $\bb X(t)$ sarting from
$\nu$. Suppose, furthermore, that in the time scale $\theta_N$ the
original process $\eta^N(t)$ spends a negligible amount of time in
$\Delta_N$:
\begin{equation*}
\lim_{N\to\infty} \mb E_{\nu_N} \Big[ \int_0^T 
\mb 1\{ \eta (s\, \theta_N) \in \Delta_N \} \, ds \Big] \;=\; 0\;.
\end{equation*}
Then, the sequence of probability measures $\mb P_{\nu_N} \circ \bb
X_N^{-1}$ converges in the soft topology to $\bb P_\nu$.
\end{theorem}

\begin{proof}
Since the sequence of probability measures $\mb P_{\nu_N} \circ (\bb
X^{\rm T}_N)^{-1}$ converges in the Skorohod topology to the measure
$\bb P_\nu$, and since in the time scale $\theta_N$ the time spent by
the chain $\eta^N(t)$ in the set $\Delta_N$ is negligible, by
Proposition 4.3 in \cite{bl2}, the sequence of measures $\mb P_{\nu_N}
\circ (\bb X^{\rm V}_N)^{-1}$ also converges in the Skorohod topology
to $\bb P_\nu$.

Denote by $P_N$ the probability measure on $D([0,T], S_L \cup \{N\})$
induced by the process $\bb X_N(t) = \Psi(\eta(t \theta_N))$ starting
from $\nu_N$. Since, by Assertion \ref{ea15}, $D([0,T], S_L \cup
\{N\})$, $D([0,T], S_L)$ are closed subsets of $E([0,T], S_{\mf d})$,
we may extend $P_N$ and $\bb P_{\nu}$ to $E([0,T], S_{\mf d})$.
Note that $\mb P_{\nu_N} \circ (\bb X^{\rm V}_N)^{-1} = P_N \circ \mf
R_L^{-1}$.  By the previous paragraph, $P_N \circ \mf R_L^{-1} \to \bb
P_\nu$ in the Skorohod topology. Since $P_N \circ \mf R_m^{-1} = P_N
\circ \mf R_L^{-1}$ for $L\le m\le N$, $P_N \circ \mf R_m^{-1} \to \bb
P_\nu = \bb P_\nu \circ \mf R_m^{-1}$ for $m\ge L$ in the Skorohod
topology. On the other hand, for $1\le m < L$, by Lemma \ref{ea06},
$P_N \circ \mf R_m^{-1} \to \bb P_\nu \circ \mf R_m^{-1}$ in the
Skorohod topology.

In conclusion, $P_N \circ \mf R_m^{-1} \to \bb P_\nu \circ \mf
R_m^{-1}$ in the Skorohod topology for all $m\ge 1$. Therefore, by
Theorem \ref{se14}, $P_N$ converges to $\bb P_\nu$ in the soft
topology, a probability measure concentrated on the closed subset
$D([0,T], S_{L})$.
\end{proof}

We may apply Theorem \ref{se15} to the zero-range processes presented
in the introduction.  Fix $x\in S_L$ and a sequence of configurations
$\eta^N$ in $\ms E^x_N$. We proved in \cite{bl3, l1} that starting
from $\Psi_N(\eta^N)$ the trace process $X^{\rm T}_N(t)$ in the time
scale $\theta_N = N^{1+\alpha}$ converges in the Skorohod topology to
a Markov chain $\bb X(t)$ and that the time spent by $\eta^N(t)$ in
the set $\Delta_N$ in the time scale $N^{1+\alpha}$ is
negligible. Therefore, by Theorem \ref{se15}, the rescaled order
parameter $\bb X_N(t)$ converges in the soft topology to $\bb X(t)$.
\smallskip

Consider now the random walk among the random traps on $\bb T^d_N$.
It follows from the results proved in \cite{jlt1, jlt2} that the trace
of $\bb X_N(t)$ on the set $\{1, \dots, L_N\}$, denoted by $\bb X^{\rm
  T}_N(t)$, where $L_N$ is a sequence which increases slowly to
$\infty$, converges in the Skorohod topology to a $K$-process $\bb
X(t)$, and that the time spent by $\eta^N(t)$ in the set $\Delta_N =
\{L_N+1, \dots, V_N\}$ in the time scale $\theta_N = v^{-1}_N$ is
negligible. By Proposition 4.3 in \cite{bl2}, the process which
records the last visit of $\bb X_N(t)$ to $\{1, \dots, L_N\}$, denoted
by $\bb X^{\rm V}_N(t)$, also converges in the Skorohod topology to
$\bb X(t)$.

Fix $j\in \bb N$ and denote by $P_N$ the probability measure on
$D([0,T], S_{V_N})$ induced by the process $\bb X_N(t) = \Psi(\eta(t
\theta_N))$ starting from $j$. By the previous paragraph, $P_N \circ
\mf R_{L_N}^{-1} \to P$ in the Skorohod topology, where $P$ is the
probability measure on $D([0,T], S_{\mf d})$ which corresponds to a
$K$-process starting from $j$. By Lemma \ref{ea06}, $P_N \circ \mf
R_m^{-1} \to P \circ \mf R_m^{-1}$ in the Skorohod topology for every
$m\ge 1$. Therefore, there exists a probability measure $P$ on the
path space $E([0,T], S_{\mf d})$, concentrated on the subset $D([0,T],
S_{\mf d})$, such that $P_N \circ \mf R_m^{-1} \to P \circ \mf R_m^{-1}$ in
the Skorohod topology for all $m\ge 1$. By Theorem \ref{se14}, $P_N$
converges to $P$ in the soft topology.

\smallskip

We conclude this section showing that the conditions (a)--(c) of
Theorem \ref{se12} follow from the convergence of the order parameter
to a Markov process and from the fact that asymptotically the process
spends a negligible amount of time on $\Delta_N$.

\smallskip\noindent{\bf 1. Random walks among traps.}  Consider the
random walk among traps $\eta(t) = \eta^N(t)$ defined in the
introduction, and recall that we denoted by $\pi_N$ the stationary
state. Fix $T>0$ and denote by $\bb Q^N_k$, $k\ge 1$, the probability
measure on $D([0,T], S_{\mf d})$ induced by the random walk $\bb
X_N(t) = \Psi_N(\eta(\theta_N t))$ starting from $k$. Note that time
has been speeded-up by $\theta_N = v_N^{-1}$, where $v_N$, defined in
\eqref{e29}, is the probability to escape from a ball of radius
$\ell_N$, and that the measure $\bb Q^N_k$ is concentrated on the set
$D([0,T], S_{V_N})$.


It is clear from this last observation that $\Lambda_T(x)=0$, $\bb
Q^N_k$--\,almost surely. On the other hand, if we denote by $\tau_j$,
$j\ge 1$, the holding times of the trajectory $x(t)$, $x (t)$ is
discontinuous at $T$ if and only if $\tau_1 + \cdots + \tau_j = T$ for
some $j$. Since, $\bb Q^N_k [\tau_1 + \cdots + \tau_j = T]=0$ for each
$j\ge 1$, $\bb Q^N_k$ is concentrated on the set $D^*([0,T], S_{\mf
  d})$.

Denote by $P_N$ the probability measure on $E([0,T], S_{\mf d})$
defined by $P_N = \bb Q^N_k \circ \mf R^{-1}_\infty$. By the last
observation, $P_N$ is concentrated on $E^*([0,T], S_{\mf d})$.  We
claim that the sequence $P_N$ fulfills all the assumptions of Theorem
\ref{se12}. We start with assumption (a). Since $\mf R_\ell x \le x$,
it is enough to show that
\begin{equation}
\label{e25}
\lim_{m\to\infty} \limsup_{N\to\infty} E_{P_N}
\Big[ \int_0^T \mb 1\{ x (s) \ge m\} \, ds \Big]
\;=\; 0\;.
\end{equation}
By definition of $P_N$, 
\begin{equation*}
\begin{split}
E_{P_N} \Big[ \int_0^T \mb 1\{ x (s) \ge m\} \, ds \Big]
\; & =\; E_{\bb Q^N_k} \Big[ \int_0^T \mb 1\{ x (s) \ge m\} \, ds \Big] \\ 
\; & \le\; \frac 1{\pi_N(k)} \sum_{j\ge 1} \pi_N(j) \, 
E_{\bb Q^N_j} \Big[ \int_0^T \mb 1\{ x (s) \ge m\} \, ds \Big]\;.  
\end{split}
\end{equation*}
Since $\pi_N$ is the stationary state, the previous sum is equal to
$T\, \pi_N\{S_{m}^c\}$, where, we recall, $S_m =\{1, \dots, m\}$. As,
for every $k\ge 1$,
\begin{equation*}
\lim_{m\to\infty} \limsup_{N\to\infty} \frac{\pi_N\{S_{m}^c\}}
{\pi_N(k)}\;=\; 0\;,
\end{equation*}
condition \eqref{e25} is in force.

We first prove Conditions (b) and (c) of Theorem \ref{se12} under
the assumption that $\theta := \sup_{N\ge 1} \theta_N$ is finite. This
is the case of the random walk on a torus $\bb T^d_N$ in dimension
$d\ge 3$. 

Since $\mf N_m(\mf R_\ell x) \le \mf N_m(x)$, $\ell\ge 1$, to prove
condition (b) of Theorem \ref{se12}, it is enough to show that for
each $m\ge 1$,
\begin{equation}
\label{e26}
\lim_{j \to\infty}\limsup_{N\to\infty} 
P_N \big[ \, \mf N_m (x) \ge j \,\big] \;=\; 0\;.
\end{equation}
The above probability is equal to $\bb Q^N_k [ \, \mf N_m (x) \ge j
\,]$.  Denote by $\tau^m_i$, $i\ge 1$, the holding times at $m$. This
is a sequence of i.i.d. mean $\theta^{-1}_N \, W_m$ exponential random
variables. Since $\{ \mf N_m (x) \ge j\} \subset \{\tau^m_1 + \cdots +
\tau^m_j\le T\}$, the previous probability is bounded by $\bb Q^N_k
\big [ \, \tau^m_1 + \cdots + \tau^m_j\le T \, \big] \;\le \; P \big [
\, T_1 + \cdots + T_j\le T \, \big]$, where $T_i$, $i\ge 1$, is a
sequence of i.i.d. mean $\theta^{-1} \, W_m$ exponential random
variables. This expression vanishes as $j\uparrow\infty$, which proves
\eqref{e26}.

In view of \eqref{e28} and since $T_{m,j}(\mf R_\ell x)$, $j\ge 1$,
are identically distributed, to prove condition (c) of Theorem
\ref{se12} we need to show that for each $m\ge 1$,
\begin{equation*}
\lim_{\epsilon\to 0} \limsup_{\ell\to\infty} \limsup_{N\to\infty} 
P_N \big[ \, T_{m,1}(\mf R_{\ell} x) < \epsilon \,\big] \;=\; 0\;.
\end{equation*}
Since $T_{m,1}({\mf R}_\ell x) \ge T_{m,1}(x)$, $\ell \ge m \ge 1$, to
prove condition (c) of Theorem \ref{se12} we just have to show that
for each $m\ge 1$,
\begin{equation}
\label{e27}
\lim_{\epsilon\to 0} \limsup_{N\to\infty} 
P_N \big[ \, T_{m,1}(x) < \epsilon \,\big] \;=\; 0\;.
\end{equation}
With the notation introduced in the previous paragraph, the
probability above is equal to $\bb Q^N_k [ \, \tau^m_{1} < \epsilon
\,]$. As $\tau^m_{1}$ is a mean $\theta^{-1}_N \, W_m$ exponential
random variable and as $\theta_N \le \theta$, the previous probability
is less than or equal to $P [ \, T < \epsilon \,]$, where $T$ is a
mean $\theta^{-1} \, W_m$ exponential random variable. This proves
condition (c) of Theorem \ref{se12} in the case where $\sup_N
\theta_N<\infty$. \smallskip

We conclude the analysis proving conditions (b) and (c) of Theorem
\ref{se12} without the assumption that $\sup_N \theta_N <
\infty$. Recall that we denote by $A_N$ the set of the first $M_N$
deepest traps, $A_N = \{x^N_1, \dots, x^N_{M_N}\}$. Choose a sequence
$M_N$ so that $M^2_N \, \ell^d_N \ll N^d$. In this case with a
probability converging to one the balls $B(x^N_i, \ell_N)$, $1\le i\le
M_N$, are disjoints.  Let $U^N_1$ be the time of the first visit to
$A_N$, $U^N_1 = \inf\{ t\ge 0 : \eta(t) \in A_N\}$, and define
recursively the sequence of stopping times $U^N_j$, $j\ge 1$, by
\begin{equation*}
U^N_{j+1} \;=\; \inf \Big \{ t\ge U^N_j : \eta(t) \in A_N \,,\,
\exists \, U^N_j\le s\le t \; \text{ s.t. } \eta(s)\not\in B_N \Big\}\;,
\end{equation*}
where $B_N = \cup_{i=1}^{M_N} B(x^N_i , \ell_N)$. Hence, the sequence
$U^N_j$ represents the successive visits to the deepest traps after
escaping from these traps.  We refer to the time interval $[U^N_j,
U^N_{j+1})$ as the $j$-th excursion.

For $m\ge 1$, let $e_1(m) = \min \{j\ge 1: \eta(U^N_j)=x^N_m\}$ be the
first excursion to the trap $x^N_m$. Define recursively $e_i(m)$,
$i\ge 1$, by
\begin{equation*}
e_{i+1}(m) = \min \{j> e_i(m): \eta(U^N_j)=x^N_m\}\;.
\end{equation*}
Note that we may have $e_{i+1}(m) = e_i(m)+1$, as the process may
escape from the trap $x^N_m$ and then return to it before visiting any
other deep trap.  We refer to $[U^N_{e_i(m)}, U^N_{e_i(m) + 1})$
as the $i$-th excursion to $x^N_m$.

Let $G^N_i$, $i\ge 1$, be the number of visits to $x^N_m$ during the
$i$-th excursion to $x^N_m$, in other words, $G^N_i$ is the number of
visits to $x^N_m$ in the time interval $[U^N_{e_i(m)},
U^N_{e_i(m)+1})$. The random variables $G^N_i$, $i\ge 1$, are
i.i.d. and have a mean $\theta_N$ geometric distribution.  Let
$T^N_{i,p}$, $1\le p\le G^N_i$, be the $p$-th holding time at $x^N_m$
in the time interval $[U^N_{e_i(m)}, U^N_{e_i(m)+1})$.  The random
variables $T^N_{i,p}$ are i.i.d. and have a mean $W_m/\theta_N$
exponential distribution.

Fix $N$ large enough for $M_N \ge \ell$ so that $\mf N_m(\mf R_\ell x)
\le \mf N_m(\mf R_{M_N} x)$. In this case, since for the trajectory
$\mf R_{M_N} x$ all visits of $\eta(t)$ to $x^N_m$ in the time
interval $[U^N_{e_i(m)}, U^N_{e_i(m) + 1})$ are counted as a single
visit,
\begin{equation*}
\{\mf N_m(\mf R_\ell x) \ge j\} \;\subset\; \Big\{
\sum_{i=1}^j \sum_{p=1}^{G^N_i} T^N_{i,p} \le T \Big\}\;.
\end{equation*}
It follows from the conclusion of the last paragraph that $\sum_{1\le
  p \le G^N_i} T^N_{i,p}$, $i\ge 1$, forms a sequence of i.i.d. mean
$W_m$ exponential random variables. This proves condition (b) of
Theorem \ref{se12}.

In view of \eqref{e28} and since $T_{m,j}(\mf R_\ell x)$, $j\ge 1$,
are identically distributed, to prove condition (c) of Theorem
\ref{se12} we need to show that for each $m\ge 1$,
\begin{equation*}
\lim_{\epsilon\to 0} \limsup_{\ell\to\infty} \limsup_{N\to\infty} 
P_N \big[ \, T_{m,1}(\mf R_{\ell} x) < \epsilon \,\big] \;=\; 0\;.
\end{equation*}
Since $T_{m,1}(\mf R_\ell x) \ge T_{m,1}(\mf R_{M_N} x)$, it is in
fact enough to show that for each $m\ge 1$,
\begin{equation*}
\lim_{\epsilon\to 0} \limsup_{N\to\infty} 
P_N \big[ \, T_{m,1}(\mf R_{M_N} x) < \epsilon \,\big] \;=\; 0\;.
\end{equation*}
This probability is equal to $\bb Q^N_k [ \, T_{m,1}(\mf R_{M_N} x) <
\epsilon \,]$ and $T_{m,1}(\mf R_{M_N} x) \ge \sum_{p=1}^{G^N_1}
T^N_{1,p}$, a mean $W_m$ exponential random variable. Therefore,
\begin{equation*}
P_N \big[ \, T_{m,1}(\mf R_{M_N} x) < \epsilon \,\big] \;\le\;
P[T < \epsilon]\;,
\end{equation*}
where $T$ is a mean $W_m$ exponential random variable, which proves
condition (c) of Theorem \ref{se12}.

\smallskip
\noindent{\bf 2. Zero-range processes.}
Consider the zero-range process $\eta(t)=\eta^N(t)$ introduced in
Section \ref{esec1}.  Denote by $\mb P_\eta$, $\eta\in E_{L,N}$, the
probability measure on the path space $D(\bb R_+, E_{L,N})$ induced by
the Markov chain $\eta(t)$ starting from $\eta$. Expectation with
respect to $\mb P_\eta$ is denoted by $\mb E_\eta$.

Fix $T>0$, $1\le x\le L$, and a sequence $\{\eta^N: N\ge 1\}$ of
configurations in $\ms E^x_N$, $\eta^N\in \ms E^x_N$. Denote by $\bb
Q^N$ the probability measure on $D([0,T], S_{\mf d})$ induced by the
process $\bb X_N(t) = \Psi_N(\eta(N^{1+\alpha} t))$ starting from
$\eta^N$. Note that time has been speeded-up by $N^{1+\alpha}$ and that
the measure $\bb Q^N$ is concentrated on the set $D([0,T], S_{N})$.

It is clear from this last observation that $\Lambda_T(x)=0$, $\bb
Q^N$--\,almost surely. On the other hand, if we denote by $\tau_j$,
$\tau^\eta_j$, $j\ge 1$, the holding times of the processes $\bb
X_N(t)$, $\eta(N^{\alpha+1} t)$, respectively, $\bb X_N (t)$ is
discontinuous at $T$ if and only if $\tau_1 + \cdots + \tau_j = T$ for
some $j$. Since, $\tau_1 + \cdots + \tau_j = \tau^\eta_1 + \cdots +
\tau^\eta_k$ for some $k\ge j$ and since $\mb P_{\eta^N} [\tau^\eta_1
+ \cdots + \tau^\eta_\ell = T]=0$ for all $\ell\ge 1$, we have that
$\bb Q^N [\tau_1 + \cdots + \tau_j = T]=0$ for each $j\ge
1$. Therefore, $\bb Q^N$ is concentrated on the set $D^*([0,T], S_{\mf
  d})$.

Denote by $P_N$ the probability measure on $E([0,T], S_{\mf d})$
defined by $P_N = \bb Q^N \circ \mf R^{-1}_\infty$. By the last
observation, $P_N$ is concentrated on $E^*([0,T], S_{\mf d})$.  We
claim that the sequence $P_N$ fulfills all the assumptions of Theorem
\ref{se12}. We start with assumption (a). As in the previous example,
it is enough to show that \eqref{e25} holds.  By definition of $P_N$,
for $N\ge m\ge L$,
\begin{equation*}
\begin{split}
E_{P_N} \Big[ \int_0^T \mb 1\{ x (s) \ge m\} \, ds \Big]
\; &=\; E_{\bb Q^N} \Big[ \int_0^T \mb 1\{ x (s) \ge m\} \, ds
\Big] \\
\; & =\; \mb E_{\eta^N} \Big[ \int_0^T 
\mb 1\{ \eta (s\, N^{\alpha +1}) \in \Delta_N \} \, ds \Big] \;.
\end{split}
\end{equation*}
Therefore, \eqref{e25} follows from assertion (M3) of \cite[Theorem
2.4]{bl3}.

We turn to condition (b) of Theorem \ref{se12}. As in the example of
random walks among traps, it is enough to prove \eqref{e26}.  Let $\ms
E_N = \ms E^1_N \cup \cdots \cup \ms E^L_N$.  Denote by $T_j$, $j\ge
1$, the holding times between successive visits to the metastable
sets: $T_1 \;=\; \inf\{t>0 : \eta (t) \in \ms E_N\}$,
\begin{equation*}
T_{j+1} \;=\; \inf\{t>0 : \eta (\mb T_j + t) \in \ms E_N
\setminus \ms E^{\mb y_j}_N \}\;, \quad \mb T_j
\;=\; T_1 + \cdots T_j\;,\quad j\ge 1\;.
\end{equation*}
where $\mb y_j = \Psi(\eta(\mb T_j))$.  Denote by $T^{\ms E}_j$, $j\ge
1$, the same sequence for the trace process $\eta^{\rm T} (t)$,
$T^{\ms E}_1 \;=\; \inf\{t>0 : \eta^{\rm T} (t) \in \ms E_N\}$.

For $1\le k\le L$, let $e_1(k) = \min \{j\ge 1: \eta(\mb T_j)\in \ms
E^k_N\}$ be the first visit to the metastable set $\ms E^k_N$. Define
recursively $e_i(k)$, $i\ge 1$, by 
\begin{equation*}
e_{i+1}(k) = \min \{j> e_i(k): \eta(\mb T_j) \in \ms E^k_N\}\;.
\end{equation*}

It is clear that $T^{\ms E}_j \le T_j$, $j\ge 1$, and that $\{\mf N_k
(\bb X_N) \ge j\} \subset \{T_{e_1(k)} + \dots + T_{e_j(k)} \le T\}
\subset \{T^{\ms E}_{e_1(k)} + \dots + T^{\ms E}_{e_j(k)} \le
T\}$. Since the sequence $T^{\ms E}_{e_j(k)}$ represents the holding
times at $k$ for the process $\bb X^{\rm T}_N(t) = \Psi_N(\eta^{\rm
  T}(N^{1+\alpha} t))$, and since the process $\bb X^{\rm T}_N(t)$
converges in the Skorohod topology to a Markov process on $\{1, \dots,
L\}$,
\begin{equation*}
\limsup_{N\to \infty} \mb P_{\eta^N} \big[ T^{\ms E}_{e_1(k)} + \dots + T^{\ms
  E}_{e_j(k)} \le T \big] \;\le\; P \big[ S_1 + \dots + S_j \le T \big]\;,
\end{equation*}
where $S_i$, $i\ge 1$, is a sequence of non-degenerate
i.i.d. exponential random variables. As $j\uparrow\infty$, this
expression vanishes, which proves \eqref{e26}.

It remains to prove assertion (c) of Theorem \ref{se12}. As argued
in the previous example, it is enough to show that \eqref{e27} holds
for every $m\ge 1$. With the notation introduced above, it means that
we have to show that
\begin{equation*}
\lim_{\epsilon\to 0} \limsup_{N\to\infty}   
\mb P_{\eta^N} \big[ T_{e_1(m)} < \epsilon \big]\;=\;0\;.
\end{equation*}
Since $T^{\ms E}_{e_1(m)}\le T_{e_1(m)}$, it is enough to prove the
previous assertion with $T^{\ms E}_{e_1(m)}$ replacing $T_{e_1(m)}$.
This follows from the convergence of $T^{\ms E}_{e_1(m)}$ to a
non-degenerate exponential distribution.

\smallskip\noindent{\bf Acknowledgements.} The author would like to
thank the referees for their comments which helped to improve the
first version of this article.

\end{document}